\documentclass[preprint]{imsart}
\usepackage{amsmath}
\usepackage{amsthm}
\usepackage{amssymb}
\RequirePackage{natbib}
\RequirePackage[colorlinks,citecolor=blue,urlcolor=blue]{hyperref}

\startlocaldefs
\newcommand{\E}{{ \operatorname E}}
\newcommand{\Cov}{{\operatorname{Cov}}}

\newcommand{\argmax}{{\operatorname{argmax}}}

\newcommand{\sargmax}{{\operatorname{sargmax}}}

\newcommand{\rank}{{\operatorname{rank}}}
\newcommand{\sign}{{\operatorname{sign}}}
\endlocaldefs

\newcommand{\ignore}[1]{}

\usepackage{longtable}
\usepackage{graphicx}
\usepackage{paralist}
\usepackage{csquotes} 
\usepackage{dsfont}
\usepackage{pdflscape}
\usepackage{morefloats}

\usepackage{nccfoots}
\usepackage{enumitem}

\theoremstyle{plain}
\newtheorem{theorem}{Theorem}
\newtheorem{proposition}{Proposition}
\newtheorem{Lem}{Lemma}

\theoremstyle{definition}
\newtheorem{assumption}{Assumption}

\theoremstyle{remark}
\newtheorem{remark}{Remark}

\usepackage{listings}
\begin{document}

\begin{frontmatter}

\title{Change point estimation based on  Wilcoxon tests in the presence of long-range dependence
}
\runtitle{Wilcoxon-based change point estimation}

\thankstext{t2}{Research supported by the German National Academic Foundation and Collaborative Research Center SFB 823 {\em Statistical modelling of nonlinear dynamic processes}.}


\author{\fnms{Annika} \snm{Betken}\ead[label=e1]{annika.betken@rub.de}}
\address{Faculty of Mathematics\\
Ruhr-Universit{\"a}t Bochum\\
 44780  Bochum,
Germany\\  \printead{e1}}

\runauthor{A. Betken}

\begin{abstract}
We consider an estimator for the location of a shift in the mean of long-range dependent sequences.  The estimation is based on  the two-sample Wilcoxon statistic. Consistency and the rate of convergence for the estimated change point are established. In the case of a constant shift height, the   $1/n$ convergence rate (with $n$ denoting the number of observations), which is typical under the assumption of independent observations, is also achieved for long memory sequences.
 It is proved that if  the change point height  decreases to $0$ with a certain rate, the suitably standardized estimator converges in distribution to a functional of a fractional Brownian motion. The estimator is tested on two well-known data sets.
Finite sample behaviors are investigated in a Monte Carlo simulation study. 
\end{abstract}




\setcounter{page}{1}
\tableofcontents

\end{frontmatter}

\section{Introduction}

Suppose that the observations $X_1, \ldots, X_n$ are generated by a stochastic process $\left(X_i\right)_{i\geq 1}$
\begin{align*}
X_i=\mu_i+Y_i,
\end{align*} 
where $(\mu_i)_{i\geq 1}$ are unknown constants and where $(Y_i)_{i\geq 1}$ is a stationary, long-range dependent (LRD, in short) process   with mean zero.
A stationary process  $(Y_i)_{i\geq 1}$ is called ``long-range dependent''  if 
its autocovariance function $\rho$, $\rho(k):=\Cov(Y_1, Y_{k+1})$,  satisfies
\begin{align}\label{autocovariances}
\rho(k)\sim k^{-D}L(k), \  \text{as } k\rightarrow \infty, 
\end{align}
where $0<D< 1$ (referred to as long-range dependence (LRD) parameter)  and where $L$ is a slowly varying function.

Furthermore, we assume that there is a change point in the mean of the observations, that is 
\begin{align*}
\mu_i=\begin{cases}
\mu,    &  \text{for} \ i=1, \ldots, k_0, \\
\mu + h_n,    & \text{for} \ i=k_0+1, \ldots, n,
\end{cases}
\end{align*}
where $k_0=\lfloor n\tau\rfloor$ denotes the change point location and  $h_n$ is the height of the level-shift.

In the following we differentiate between fixed and local changes. Under fixed changes we assume that $h_n=h$ for some $h\neq 0$.
Local changes are characterized  by a sequence $h_n$, $n\in \mathbb{N}$, with $h_n\longrightarrow 0$ as $n\longrightarrow\infty$;   in other words, in a model where the height of the jump  decreases with increasing sample size
$n$.

In order to test the hypothesis
\begin{align*}
H: \mu_1=\ldots =\mu_n
\end{align*}
against the alternative 
\begin{align*}
A: \mu_1=\ldots =\mu_k\neq \mu_{k+1}=\ldots =\mu_n \ \text{for  some $k\in \left\{1, \ldots, n-1\right\}$}
\end{align*}
 the Wilcoxon change point test  can be applied. It rejects the hypothesis for large values of the Wilcoxon test statistic defined by
\begin{align*}
W_n:=\max\limits_{1\leq k\leq n-1}\left|W_{k, n}\right|,  \
\text{where} \  W_{k, n}:=\sum\limits_{i=1}^k\sum\limits_{j=k+1}^n\left(1_{\left\{X_i\leq X_j\right\}}-\frac{1}{2}\right)
\end{align*}
(see \cite{DehlingRoochTaqqu2013a}).
Under the assumption that there is a change point in the mean in $k_0$
we expect the absolute value of $W_{k_0, n}$ to exceed the absolute value of $W_{l, n}$ for any $l\neq k_0$.
Therefore, it seems natural to define an
estimator of $k_0$ by 
\begin{align*}
\hat{k}_W=\hat{k}_W(n):=\min \left\{k : \left|W_{k, n}\right|=\max\limits_{1\leq i\leq n-1}\left|W_{i, n}\right|\right\}.
\end{align*} 
Preceding papers that address the problem of estimating  change point locations in dependent observations $X_1, \ldots, X_n$ with a shift in mean often refer to a family of  estimators  based on the CUSUM change point test statistics $C_{n}(\gamma):=\max_{1\leq k\leq n-1}|C_{k, n}(\gamma)|$, where 
\begin{align*}
C_{k, n}(\gamma):=\left(\frac{k(n-k)}{n}\right)^{1-\gamma}\left(\frac{1}{k}\sum\limits_{i=1}^kX_i-\frac{1}{n-k}\sum\limits_{i=k+1}^n X_i\right)
\end{align*}
with parameter $0\leq \gamma<1$.
The corresponding change point estimator is defined by
\begin{align}\label{eq:CUSUM_estimator}
\hat{k}_{C, \gamma}=\hat{k}_{C, \gamma}(n):=\min \left\{k : \left|C_{k, n}(\gamma)\right|=\max\limits_{1\leq i\leq n-1}\left|C_{i, n}(\gamma)\right|\right\}.
\end{align}
For long-range dependent Gaussian processes \cite{HorvathKokoszka1997} derive the asymptotic distribution of the estimator $\hat{k}_{C, \gamma}$ under the assumption of a decreasing jump height $h_n$, i.e. under the assumption that $h_n$ approaches $0$ as the sample size $n$ increases. Under  non-restrictive constraints on the dependence structure of the data-generating process (including long-range dependent time series) \cite{KokoszkaLeipus1998} prove consistency of $\hat{k}_{C, \gamma}$ under the assumption of fixed as well as  decreasing jump heights. Furthermore, they establish the convergence rate of the change point estimator as a function of the intensity of dependence in the data if the jump height is constant.
\cite{HarizWylie2005} show that under a similar assumption on the decay of the autocovariances the convergence rate that is achieved in the case of independent observations can be obtained for short- and long-range dependent data, as well. Furthermore, it is shown in their paper that for a decreasing jump height the convergence rate derived by  \cite{HorvathKokoszka1997}  under the assumption of gaussianity can also be established under more general assumptions on the data-generating sequences. 

\cite{Bai1994} establishes an estimator for the location of a shift in the mean by the method of least squares.
He proves consistency, determines the rate of convergence of the change point estimator and derives its asymptotic distribution. These results are shown to hold for weakly dependent observations that satisfy a linear model  and cover, for example,   ARMA($p$, $q$)-processes.
Bai extended these results to the estimation of the location of a parameter change in multiple regression models that also allow for lagged dependent variables and trending regressors (see \cite{Bai1997}).
A generalization of these results to  possibly long-range dependent data-generating processes (including fractionally integrated processes) is given in \cite{KuanHsu1998} and \cite{LavielleMoulines2000}.
Under the assumption of independent data \cite{Darkhovskh1976}  establishes an estimator for the location of a change in distribution based on the two-sample Mann-Whitney test statistic. He obtains a   convergence rate  that has order $\frac{1}{n}$, where $n$ is the number of observations.
Allowing for strong dependence in the data 
\cite{GiraitisLeipusSurgailis1996} consider  Kolmogorov-Smirnov and Cram\'{e}r-von-Mises-type  test statistics for the detection of a change in the marginal distribution of the random variables that underlie the observed data.
Consistency of the corresponding change point estimators is proved under the assumption that the jump height approaches $0$.
A change point estimator based on a self-normalized CUSUM test statistic has been applied in \cite{Shao2011} to real data sets. Although Shao assumes validity of using the estimator,  the article does not cover a formal proof of consistency.  Furthermore, it has been noted by \cite{ShaoZhang2010} that even under the assumption of short-range dependence it seems difficult to obtain the asymptotic distribution of the estimate. 

In this paper we  shortly address the issue of estimating the change point location on the basis of the self-normalized Wilcoxon test statistic proposed in \cite{Betken2016}.

In order to construct the self-normalized Wilcoxon test statistic, we have to consider the ranks $R_i$, $i=1,\ldots,n$, of the observations $X_1, \ldots, X_n$.
These  are defined by
$R_i:=\rank(X_i)=\sum_{j=1}^n1_{\{X_j\leq X_i\}}$ for $i=1,\ldots,n$. The self-normalized two-sample test statistic is defined by
\begin{equation*}
SW_{k, n}
=\frac{\sum_{i=1}^kR_i-\frac{k}{n}\sum_{i=1}^nR_i}{\bigg\{\frac{1}{n}\sum_{t=1}^k S_t^2(1,k)+\frac{1}{n}\sum_{t=k+1}^n S_t^2(k+1,n)\bigg\}^{1/2}}, 
\end{equation*}
where 
\begin{equation*}
S_{t}(j, k):=\sum\limits_{h=j}^t\left(R_h-\bar{R}_{j, k}\right)\ \ \text{with }\bar{R}_{j, k}:=\frac{1}{k-j+1}\sum\limits_{t=j}^kR_t.
\end{equation*}
The self-normalized Wilcoxon change point test for the test problem $(H, A)$ rejects the hypothesis for large values of $T_n(\tau_1, \tau_2)=\max_{k\in \left\{\lfloor n\tau_1\rfloor, \ldots,  \lfloor n\tau_2\rfloor\right\}}\left|SW_{k, n}\right|$, where $0< \tau_1 <\tau_2 <1$. Note that the proportion of the data that is included in the calculation of the supremum  is restricted by $\tau_1$ and $\tau_2$. A common  choice for these parameters is $\tau_1= 1-\tau_2=0.15$; see  \cite{Andrews1993}.

A natural change point estimator that results from the self-normalized Wilcoxon test statistic is 
\begin{align*}
\hat{k}_{SW}=\hat{k}_{SW}(n):=\min \left\{k : \left|
SW_{k, n}\right|=\max\limits_{\lfloor n\tau_1\rfloor\leq i\leq \lfloor n\tau_2\rfloor}\left|SW_{i, n}\right|\right\}.
\end{align*} 
We will prove consistency of the estimator $\hat{k}_{SW}$ under fixed changes and under local changes whose height converges to $0$ with a rate depending on the intensity of dependence in the data. 
Nonetheless, the main aim of this paper is to characterize the asymptotic behavior of the change point estimator $\hat{k}_W$.   In Section \ref{Main Results} we  establish consistency of $\hat{k}_W$ and $\hat{k}_{SW}$, derive the optimal convergence rate of $\hat{k}_W$ and finally consider its asymptotic distribution.
Applications to two well-known data sets can be found in Section \ref{Applications}. The finite sample properties of the estimators are investigated  by simulations in Section \ref{Simulations}. Proofs of the theoretical results are given in Section \ref{Proofs}.

\section{Main Results}\label{Main Results}

Recall that for  fixed $x$, $x \in \mathbb{R}$, the 
Hermite expansion of $1_{\left\{G(\xi_i)\leq x\right\}}-F(x)$ is given by
\begin{align*}
1_{\left\{G(\xi_i)\leq x\right\}}-F(x)=\sum\limits_{q=1}^{\infty}\frac{J_q(x)}{q !}H_q(\xi_i),
\end{align*}
where $H_q$ denotes the $q$-th order Hermite polynomial and where
\begin{align*}
J_q(x)=\E \left(1_{\left\{G(\xi_i)\leq x\right\}}H_q(\xi_i)\right).
\end{align*}

\begin{assumption}\label{ass:subordination} Let $Y_i=G(\xi_i)$, where $\left(\xi_i\right)_{i\geq 1}$  is a stationary, long-range dependent Gaussian  process   with mean $0$, variance $1$ and LRD parameter $D$. We assume that $0<D <\frac{1}{r}$, where 
 $r$ denotes the Hermite rank of the class of functions $1_{\left\{G(\xi_i)\leq x\right\}}-F(x)$, $x \in \mathbb{R}$, 
 defined by
\begin{align*}
r:=\min \left\{q\geq 1: J_q(x)\neq 0 \ \text{for  some} \ x\in \mathbb{R}\right\}.
\end{align*}
Moreover, we assume that $G:\mathbb{R}\longrightarrow \mathbb{R}$ is a measurable function and that $\left(
Y_i\right)_{i\geq 1}$ has a  continuous distribution function $F$.
\end{assumption}

Let
\begin{align*} 
g_{D, r}(t):=t^{\frac{rD}{2}}L^{-\frac{r}{2}}(t)
\end{align*} and define
\begin{align*}
d_{n, r} :=\frac{n}{g_{D, r}(n)}c_{r}, \
\text{where} \ c_{r}:=\sqrt{\frac{2 r !}{(1-Dr)(2-Dr)}}.
\end{align*}
Since $g_{D, r}$ is a regularly varying function, there exists a function $g_{D, r}^{-}$ such that  
\begin{align*}
g_{D, r}( g_{D, r}^{-}(t))\sim g_{D, r}^{-}( g_{D, r}(t))\sim t, \text{ as $t\rightarrow\infty$,}
\end{align*}
(see Theorem 1.5.12  in  \cite{Bingham1987}).
We refer to $g_{D, r}^{-}$ as the asymptotic inverse of $g_{D, r}$.

The following result states that $\frac{\hat{k}_W}{n}$ and $\frac{\hat{k}_{SW}}{n}$ are consistent estimators for the change point location under fixed as well as certain local changes.

\begin{proposition}\label{Prop:consistency}
Suppose that Assumption \ref{ass:subordination} holds.
Under fixed changes, $\frac{\hat{k}_W}{n}$ and $\frac{\hat{k}_{SW}}{n}$ are consistent estimators for the change point location. The estimators are also consistent under local changes if $h_n^{-1}=o\left(\frac{n}{d_{n, r}}\right)$ and  if $F$ has a bounded
density $f$. In  other words,  we have
\begin{align*}
\frac{\hat{k}_W}{n}\overset{P}{\longrightarrow}\tau, \qquad
\frac{\hat{k}_{SW}}{n}\overset{P}{\longrightarrow}\tau
\end{align*}
in both situations.
Furthermore, it follows that the  Wilcoxon test is consistent under these assumptions (in the sense that $\frac{1}{nd_{n, r}}\max_{1 \leq k\leq n-1}|W_{k, n}|\overset{P}{\longrightarrow}\infty$).
\end{proposition}

The following theorem establishes a convergence rate for the change point estimator $\hat{k}_W$. Note that only  under local changes the convergence rate depends on the intensity of dependence in the data.

\begin{theorem}\label{convergence rate}
Suppose that Assumption \ref{ass:subordination} holds  and let $m_n:=g_{D, r}^{-}(h_n^{-1})$. Then, we have
\begin{align*}
\left|\hat{k}_W-k_0\right|=\mathcal{O}_P(m_n)
\end{align*}
if either 
\begin{itemize}
\item  $h_n =  h$ with $h\neq 0$ 
\end{itemize}
or
\begin{itemize}
\item $\lim_{n\rightarrow \infty}h_n=0$ with $h_n^{-1}=o\left(\frac{n}{d_{n, r}}\right)$ and   $F$ has a bounded
density $f$.
\end{itemize}
\end{theorem}

\begin{remark}\label{remark:convergence_rate}
\leavevmode
\begin{enumerate}
\item Under fixed changes $m_n$ is constant. As a consequence, $|\hat{k}_W-k_0|=\mathcal{O}_P(1)$. This result corresponds to the convergence rates obtained by \cite{HarizWylie2005} for the CUSUM-test based change point estimator and by \cite{LavielleMoulines2000} for the least-squares estimate of the change point location. Surprisingly, in this case the rate of convergence is independent of the intensity of dependence in the data characterized by the value of the LRD parameter $D$. An explanation for this phenomenon  might be the occurrence of two opposing effects: increasing values of the LRD parameter $D$ go along with a slower convergence of the test statistic $W_{k, n}$ (making estimation more difficult), but a more regular behavior of the random component (making estimation easier) (see \cite{HarizWylie2005}).
\item Note that if  $h_n^{-1}=o\left(\frac{n}{d_{n, r}}\right)$ and $m_n=g_{D, r}^{-}(h_n^{-1})$, it holds that
\begin{itemize}
\item $m_n\longrightarrow \infty$,
\item $\frac{m_n}{n}\longrightarrow 0$,
\item $ \frac{d_{m_n, r}}{m_n} \sim h_n$,
\end{itemize}
as $n\longrightarrow \infty$.
\end{enumerate}
\end{remark}

Based on the previous results it is possible to derive the asymptotic distribution of the change point estimator $\hat{k}_W$:

\begin{theorem}\label{thm:asymp_distr}
Suppose that Assumption \ref{ass:subordination} holds  with $r=1$ and  assume that $F$ has a bounded
density $f$. Let $m_n:=g_{D, 1}^{-}(h_n^{-1})$, let $B_H$ denote a fractional Brownian motion process  and
define $h(s; \tau)$ by
\begin{align*}
h(s; \tau)=
\begin{cases}
s(1-\tau)\int f^2(x)dx  &\text{if $s\leq 0$}\\
-s\tau \int f^2(x)dx  &\text{if $s> 0$}
\end{cases}.
\end{align*}. If
$h_n^{-1}=o\left(\frac{n}{d_{n, 1}}\right)$,
then, for all $M>0$,
\begin{align*}
\frac{1}{e_n}\left(W_{k_0+\lfloor m_n s\rfloor, n}^2-W_{k_0, n}^2\right), \ -M\leq s\leq M,
\end{align*}
with  $e_n=n^3h_nd_{m_n, 1}$, converges in distribution to
\begin{align*}
2\tau(1-\tau)\int f^2(x)dx\left(\sign(s)B_H(s)\int J_1(x)dF(x)+h(s; \tau)\right), \ -M\leq s\leq M,
\end{align*}
in the Skorohod space $D\left[-M, M\right]$. 
Furthermore, it follows that
$m_n^{-1}(\hat{k}_W-k_0)$
converges in distribution to
\begin{align}\label{eq:limit_cpe}
\argmax_{-\infty < s <\infty}\left(\sign(s)B_H(s)\int J_1(x)dF(x)+h(s; \tau)\right).
\end{align}
\end{theorem}

\begin{remark}
\leavevmode
\begin{enumerate}
\item Under local changes the assumption on $h_n$
is equivalent to Assumption C.5 (i) in \cite{HorvathKokoszka1997}. Moreover, the limit distribution \eqref{eq:limit_cpe} closely resembles the limit distribution of the CUSUM-based change point estimator considered in that paper.
\item
The proof of Theorem \ref{thm:asymp_distr} is mainly based on 
the
 empirical process non-central limit theorem for subordinated Gaussian sequences
in \cite{DehlingTaqqu1989}. The sequential empirical process has also been studied by many other authors in the context of different models. See, among
many others, the following:
\cite{Muller1970} and \cite{Kiefer1972} for independent and identically distributed data,
\cite{BerkesPhillip1977} and
\cite{PhilippPinzur1980} for strongly mixing processes,
\cite{BerkesHoermannSchauer2009} for S-mixing processes, \cite{GiraitisSurgailis1999} for long memory linear (or moving average) processes,  \cite{DehlingDurieuTusche2014} for multiple mixing processes.
Presumably, in these situations
the asymptotic distribution of $\hat{k}_W$ can be derived by the same argument as in the proof of Theorem  \ref{thm:asymp_distr} for subordinated Gaussian processes.
In particular,   Theorem 1 in \cite{GiraitisSurgailis1999} can be considered as a generalization of Theorem 1.1 in \cite{DehlingTaqqu1989}, i.e.  with an appropriate normalization the change point estimator $\hat{k}_W$, computed with respect to
 long-range dependent  linear processes as defined in \cite{GiraitisSurgailis1999}, should converge in distribution to a limit that  corresponds to \eqref{eq:limit_cpe} (up to multiplicative constants).
\end{enumerate}
\end{remark}

\section{Applications}\label{Applications}

We consider two well-known data sets which have been analyzed before. 
We compute the estimator $\hat{k}_W$ based on the given observations and put our results into context with the findings and conclusions of other authors.

\begin{figure}[ht]
\includegraphics[scale=0.4]{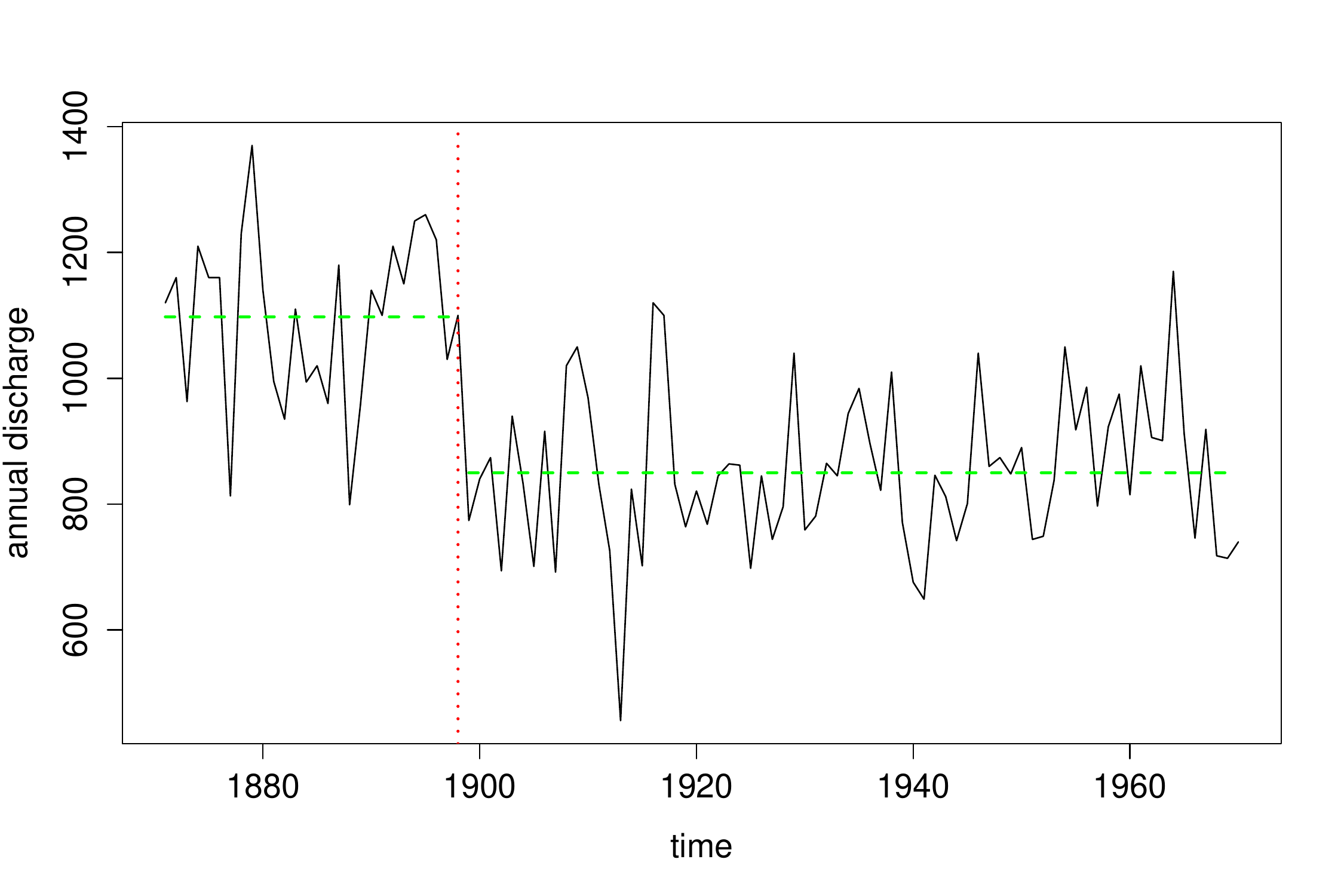}
\caption{Measurements of the annual discharge of the river Nile at  Aswan in $10^8 m^3$ for the years 1871-1970.  The dotted line indicates the potential change point estimated by $\hat{k}_{\text{W}}$; the dashed lines designate the sample means for the pre-break and post-break samples.}
\label{Nile}
\end{figure}

The plot in Figure \ref{Nile} depicts 
 the annual volume of
discharge from the Nile river at Aswan in $10^8 m^3$ for the years 1871 to 1970. The data set is included in any standard distribution of \lstinline$R$.
Amongst others, \cite{Cobb1978}, \cite{MacNeill1991},  \cite{WuZhao2007},  \cite{Shao2011} and \cite{BetkenWendler2016} provide statistically significant evidence for a decrease of the Nile's annual discharge 	towards the end of the 19th century. 

The  construction of the Aswan Low Dam between 1898 and 1902  serves as a popular explanation for an abrupt change in the data around the turn of the century.
Yet, Cobb gave another explanation for the decrease in water volume by citing rainfall records which suggest a decline of tropical rainfall at that time.
In fact, an application of the change point estimator $\hat{k}_W$ identifies a change in 1898. This result seems to be in good accordance with the estimated change point locations suggested by  other authors:
Cobb's analysis of the Nile data  leads to the conjecture of a significant decrease in discharge volume in 1898. Moreover, computation of the CUSUM-based change point estimator $\hat{k}_{C, 0}$ considered in  \cite{HorvathKokoszka1997} indicates a change in 1898.
 \cite{Balke1993} and \cite{WuZhao2007} 
suggest that the change occurred in 1899.

\begin{figure}[ht]
\includegraphics[scale=0.4]{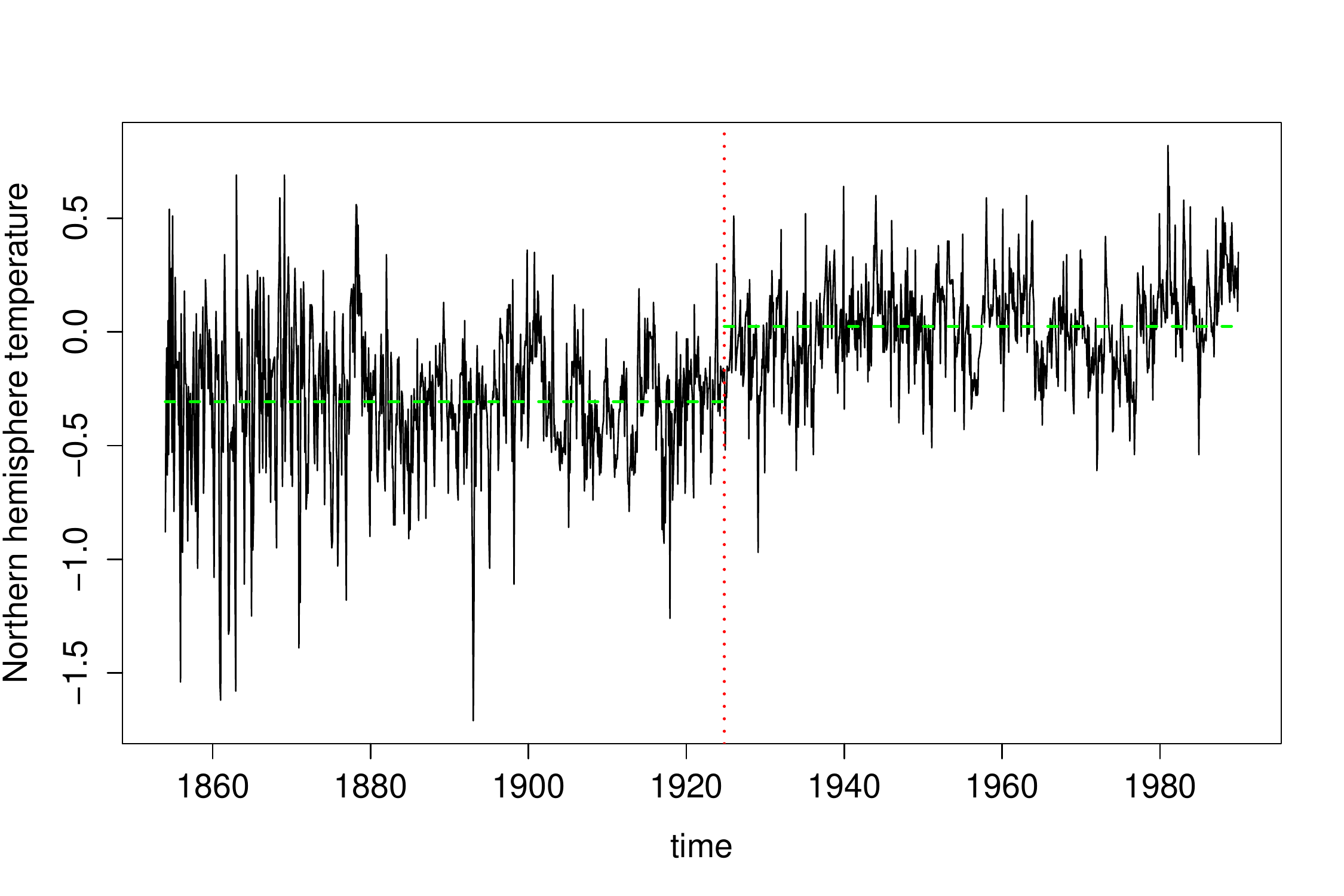}
\caption{Monthly temperature of the Northern hemisphere for the years 1854-1989 from the data base held at the Climate Research Unit of the University of East Anglia, Norwich, England. The temperature anomalies (in degrees C) are calculated with respect to the reference period 1950-1979.  The dotted line indicates the location of the potential change point; the dashed lines designate the sample means for the pre-break and post-break samples.}
\label{NemiTemp}
\end{figure}

The second data set consists of the seasonally adjusted monthly deviations of the  temperature (degrees C) for the Northern hemisphere during the years 1854 to 1989 from the monthly averages over the period 1950 to 1979. The data has been taken from  the \lstinline$longmemo$ package in \lstinline$R$. It results from spatial averaging of temperatures measured over land and sea.
In view of the plot in Figure \ref{NemiTemp} it seems natural to assume that the data generating process is non-stationary.
Previous analysis of this data offers different explanations for the  	irregular behavior of the time series.   \cite{DeoHurvich1998}  fitted a linear trend to the data, thereby providing statistical evidence 
for global warming during the last decades.
However, the consideration of a more general stochastic model by the assumption of so-called semiparametric fractional autoregressive  (SEMIFAR) processes in \cite{BeranFeng2002} does not confirm the conjecture of a trend-like behavior. Neither does the investigation of the global temperature data  in   \cite{Wang2007} support the hypothesis of an increasing trend.
It is pointed out by Wang that the trend-like behavior of the Northern hemisphere temperature data may have been generated by  stationary long-range dependent processes. 
Yet, it is shown in  \cite{Shao2011} and also in \cite{BetkenWendler2016} that  under model assumptions that include long-range dependence an application of change point tests leads to a rejection of the hypothesis that the time series is stationary.
According to  \cite{Shao2011} an estimation  based on a self-normalized CUSUM test statistic suggests a change around October 1924. Computation of the change point estimator $\hat{k}_W$ corresponds to a  change point located around June 1924. The same change point location results from an application of the previously mentioned estimator $\hat{k}_{C, 0}$ considered in  \cite{HorvathKokoszka1997}. In this regard estimation by  $\hat{k}_W$ seems to be in good accordance with the results of alternative change point estimators.

\section{Simulations}\label{Simulations}

We will now investigate the finite sample performance of the change point estimator $\hat{k}_W$ and compare it to corresponding simulation results for the estimators $\hat{k}_{SW}$ (based on the self-normalized Wilcoxon test statistic) and $\hat{k}_{\text{C}, 0}$  (based on the CUSUM test statistic with parameter $\gamma = 0$) .
For this purpose, we consider two different scenarios:
\begin{enumerate}
\item
Normal margins: We generate  fractional Gaussian noise time series $(\xi_i)_{i\geq 1}$ and choose $G(t) = t$ in Assumption \ref{ass:subordination}.
As a result, the simulated observations $\left(Y_i\right)_{i\geq 1}$ are Gaussian with autocovariance function $\rho$ satisfying
\begin{align*}
\rho(k)\sim \left(1-\frac{D}{2}\right)\left(1-D\right)k^{-D}.
\end{align*}
Note that  in this case the Hermite coefficient $J_1(x)$ is not equal to $0$ for all $x\in \mathbb{R}$  (see \cite{DehlingRoochTaqqu2013a}) so that $m = 1$, where $m$ denotes the Hermite rank of $1_{\left\{G(\xi_i)\leq x\right\}}-F(x), x\in \mathbb{R}$.  
Therefore, Assumption \ref{ass:subordination} holds for all values of $D\in\left(0, 1\right)$.
\item Pareto margins:
In order  to get standardized Pareto-distributed data which has a representation as a functional of a Gaussian process, we consider the transformation 
\begin{align*}
G(t)=\left(\frac{\beta k^2}{(\beta -1)^2(\beta-2)}\right)^{-\frac{1}{2}}\left(k(\Phi(t))^{-\frac{1}{\beta}} -\frac{\beta k}{\beta -1}\right) 
\end{align*}
with parameters $k, \beta>0$ and with $\Phi$ denoting the standard normal distribution function.
Since $G$ is a strictly decreasing function, it follows  by Theorem 2 in \cite{DehlingRoochTaqqu2013a}  that 
the Hermite rank of 
$1_{\left\{G(\xi_i)\leq x\right\}}-F(x), x\in \mathbb{R}$,
is $m = 1$ so that Assumption \ref{ass:subordination} holds for all values of $D\in\left(0, 1\right)$.
\end{enumerate}
 To analyze the behavior of the estimators we simulated $500$ time series of length $600$ and added a level shift of height $h$ 
after a proportion $\tau$ of the data. We have done so for several choices of  $h$ and $\tau$.
The descriptive statistics, i.e. mean, sample standard deviation (S.D.) and quartiles, are reported in Tables \ref{sampling_distribution_Wilcoxon}, \ref{sampling_distribution_SN_Wilcoxon}, and \ref{sampling_distribution_CUSUM}  for the three change point estimators $\hat{k}_W$, $\hat{k}_{SW}$ and $\hat{k}_{C, 0}$. 

The following observations, made on the basis of Tables \ref{sampling_distribution_Wilcoxon}, \ref{sampling_distribution_SN_Wilcoxon}, and \ref{sampling_distribution_CUSUM}, correspond  to the expected behavior of consistent change point estimators: 
\begin{itemize}
\item Bias and variance of the estimated change point location  decrease when  the height  of the level shift increases. 
\item Estimation of the time of change is more accurate  for  breakpoints located in the middle of the sample than estimation of change point locations that lie close   to the boundary of the testing region. 
\item High values of $H$ go along with an increase of bias and variance. This seems natural since when there is very strong dependence, i.e. $H$
is large, the variance of the series  increases, so that it becomes  harder to accurately estimate the location of a level shift.
\end{itemize}
A comparison of the descriptive statistics of the estimator $\hat{k}_W$
(based on the Wilcoxon statistic)  and $\hat{k}_{SW}$ (based on the  self-normalized Wilcoxon statistic)
shows that:
\begin{itemize}
\item In most cases the estimator $\hat{k}_{SW}$ has a smaller bias, especially for an early change point location. Nevertheless,  the difference between the biases of  $\hat{k}_{SW}$ and $\hat{k}_W$  is not   big.
\item In general the sample standard deviation of $\hat{k}_W$ is smaller than that of $\hat{k}_{SW}$. Indeed, it is only slightly better for $\tau=0.25$, but there is a clear difference  for $\tau=0.5$.   
\end{itemize}

All in all, our simulations do not give rise to choosing $\hat{k}_{SW}$ over $\hat{k}_W$. In particular,   better standard deviations of $\hat{k}_W$ compensate  for  smaller biases of $\hat{k}_{SW}$.

Comparing the finite sample performance of $\hat{k}_W$ and the CUSUM-based change point estimator $\hat{k}_{C, 0}$ we make the following observations:
\begin{itemize}
\item For fractional Gaussian noise time series bias and variance of $\hat{k}_{C, 0}$ tend to be slightly better, at least when $\tau=0.25$ and  especially for relatively high level shifts. Nonetheless, the deviations are in most cases negligible.
\item If the change happens in the middle of a sample with normal margins,  bias and variance of  $\hat{k}_W$ tend to be smaller, especially for relatively high level shifts. Again, in most cases the deviations are negligible.
\item For Pareto($3$, $1$) time series $\hat{k}_W$ clearly outperforms $\hat{k}_{C, 0}$ by yielding smaller biases and decisively smaller variances for almost every combination of parameters that has been considered. 
  The performance of the estimator $\hat{k}_{C, 0}$ surpasses the performance of $\hat{k}_W$  only for high values of the jump height $h$. 
\end{itemize}
It is well-known that the  Wilcoxon change point test is more robust against outliers in  data sets than the CUSUM-like  change point  tests, i.e. the Wilcoxon test  outperforms CUSUM-like tests if heavy-tailed time  series are considered.
Our simulations confirm that this observation is also reflected by the finite sample behavior of the corresponding change point estimators.

\begin{figure}[ht]
\includegraphics[scale=0.4]{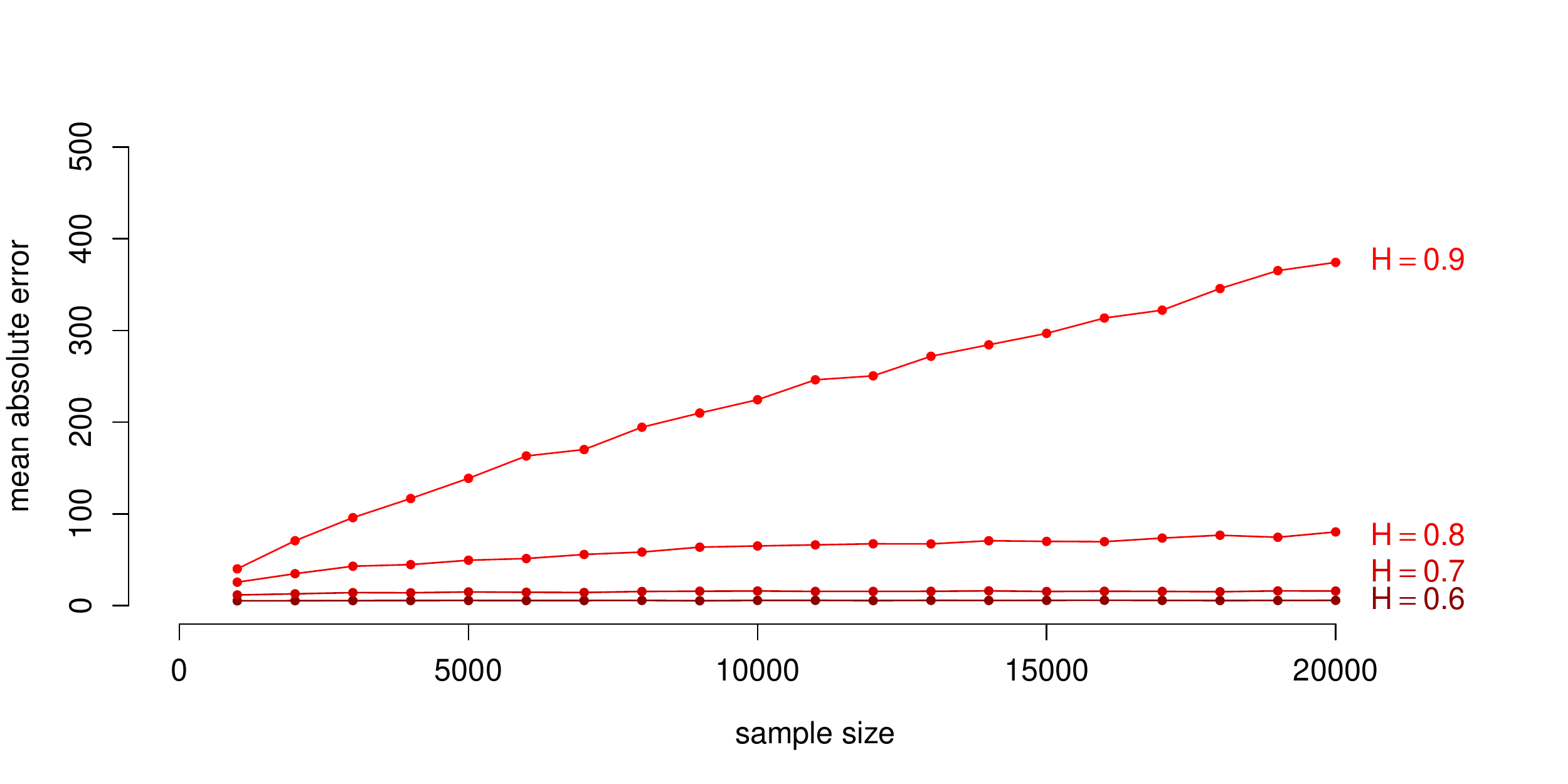}
\caption{The MAE of $\hat{k}_W$ for different values of $H$.}
\label{mean_absolute_error}
\end{figure}

As noted in Remark \ref{remark:convergence_rate}, $\hat{k}_W-k_0=\mathcal{O}_P(1)$
under the assumption of a constant change point height $h$.
This observation is illustrated by simulations of the mean absolute error
\begin{align*}
\text{MAE}=\frac{1}{m}\sum\limits_{i=1}^{m}\left|\hat{k}_{W, i}-k_0\right|,
\end{align*}
where $\hat{k}_{W, i}$, $i=1, \ldots, m$, denote the estimates  for $k_0$, computed on the basis of $m=5000$ different sequences of fractional Gaussian noise time series. 

Figure \ref{mean_absolute_error} depicts a plot of $\text{MAE}$ against the sample size $n$ with $n$ varying between $1000$ and $20000$.

Since $\hat{k}_W-k_0=\mathcal{O}_P(1)$ due to Theorem \ref{convergence rate}, we expect $\text{MAE}$ to approach a constant as $n$ tends to infinity. This can be clearly seen in Figure  \ref{mean_absolute_error} for $H\in\left\{0.6, 0.7, 0.8\right\}$.
For a high intensity of dependence in the data (characterized by $H=0.9$) convergence  becomes slower. This is due  to a slower convergence of the test statistic $W_{n}(k)$ which, in finite samples, is not canceled out by the effect  of a more regular behavior of the sample paths of the limit process.

\begin{landscape}
\begin{table}
\caption{Descriptive statistics of the sampling distribution of $\hat{k}_W$ for a change in the mean based on $500$ fractional Gaussian noise and Pareto
time series of length $600$ with Hurst parameter $H$ and a change in mean in $\tau$ of height $h$. }
\begin{tabular}{llllrrrr}
margins & $\tau$ & $h$ &  & $H= 0.6$ & $H = 0.7$ & $H = 0.8$ & $H = 0.9$ \\ 
\hline
normal & $0.25$ & $0.5$ & mean (S.D.)& {193.840} (64.020) & {227.590} (99.788) & {252.408} (110.084)& {270.646} (113.720)\\  
& &  & quartiles & (150, 168, 217.25)& (150, 191,  284.25)& (157, 226.5, 335.25)&  (172.75, 250, 353) 
\\ 
& & $1$ & mean (S.D.) & 164.244 (27.156)& 176.362 (42.059)&{188.328} (63.751)&{215.108} (88.621)\\  
&&  & quartiles  & (150, 153.5, 167)&  (150,  158,  190)& (150,  159.5, 206.25)&  (150  176  256)  \\ 
& & $2$ & mean (S.D.) &{153.604} (8.255) &{156.656} (12.393)&{164.338} (29.570)&{173.610} (41.514)\\ 
& &  & quartiles &  (150,  151,  154)&  (150,  151,  158)&  (150,  151,  164)& (150, 152, 180.25)\\ 
 &  &  &  &  &  &  \\ 
& $0.5$ & $0.5$ & mean (S.D.)&{299.506} (30.586) &{301.870} (61.392)&{300.774} (82.610)&{298.930} (98.368)\\ 
& &  & quartiles   &  (291,  300,  309)& (274.75, 300.5, 320.25)& (264, 299, 339.25)&  (233,  299,  353)\\ 
 & & $1$ & mean (S.D.) &{300.014} (9.141)&{300.438} (18.695)&{302.592} (42.213)&{300.902} (50.487)\\ 
& &  & quartiles   & (298,  300,  302)&  (297,  300,  304)&  (293,  300  307)&  (290,  300,  311)\\ 
& & $2$ & mean (S.D.) &{300.064} (1.294) &{299.922} (3.215)&{299.504} (5.520)&{300.282} (7.494)\\  
& &  & quartiles   & (300,  300,  300)&  (300,  300,  300)&  (300,  300,  300)&  (300,  300,  300)\\ 
 &  &  &  &  &  &  \\ 
  &  &  &  &  &  &  \\ 
Pareto$(3,1)$ & $0.25$ & $0.5$ & mean (S.D.)&{158.166} (17.762)&{164.080} (31.219)&{179.512} (58.871)&{194.126} (74.767)\\ 
& &  & quartiles & (150, 151, 159.25)& (150, 152, 168)& (150, 154, 191.25)& (150, 159, 218.25)\\  
& & $1$ & mean (S.D.) &{154.160} (8.765)&{156.090} (13.516)&{164.712} (28.774)&{178.174} (54.429)\\  
& &  & quartiles & (150, 151, 155)& (150, 151, 157)& (150, 152, 168)& (150, 152, 186)\\ 
& & $2$ & mean (S.D.) &{152.256} (4.852)&{155.592}  (11.092)&{160.686} (24.599)&{169.374} (38.197)\\ 
& &  & quartiles & (150, 150, 152)& (150, 151, 155.25)& (150, 151, 159)& (150, 150, 172)\\ 
 &  &  &  &  &  &  \\ 
& $0.5$ & $0.5$ & mean (S.D.)&{298.072} (6.008)&{296.432} (13.441)&{293.060} (26.221)&{289.946} (45.739)\\ 
 & &  & quartiles & (297, 300, 300)& (296, 300, 300)& (294, 300, 301)& (291, 300, 301)\\ 
& & $1$ & mean (S.D.)&{299.178} (2.712)&{298.744} (4.587)&{296.674} (11.585)&{296.168} (20.424)\\ 
& &  & quartiles  & (299, 300, 300)& (299, 300, 300)& (298, 300, 300)& (300, 300, 300)\\  
& & $2$ & mean (S.D.)& 299.798 (1.008)& 299.716 (1.543)& 299.384 (3.070)&298.896 (6.560)\\ 
& &  & quartiles & (300, 300, 300)& (300, 300, 300)& (300, 300, 300)& (300, 300, 300) 
\end{tabular}
\label{sampling_distribution_Wilcoxon}
\end{table}
\end{landscape}

\begin{landscape}	
\begin{table}
\caption{Descriptive statistics of the sampling distribution of $\hat{k}_{SW}$ for a change in the mean based on $500$ replications of fractional Gaussian noise and Pareto time series of length $600$ with Hurst parameter $H$ and a change in mean in $\tau$ of height $h$. }
\begin{tabular}{l l l l r r r r}
 margins & $\tau$ & $h$ &  & $H= 0.6$ & $H = 0.7$ & $H = 0.8$ & $H = 0.9$ \\ 
\hline
normal & $0.25$ & $0.5$ & mean (S.D.) &{172.288} (63.639)&{216.934} (110.934) &{242.202} (119.655) &{268.878} (122.615)\\
& &  & quartiles  & (135, 153, 183.25) & (138, 171, 272.5) & (143, 207.5, 333.5) & (157, 243.5, 370.25) \\
& & $1$ & mean (S.D.)&{152.406} (24.840) &{160.618} (39.834)&{174.424} (70.673) &{204.906} (99.648)\\ 
& &  & quartiles & (140, 149, 158) & (139, 150.5, 172.25) & (136, 150, 188.25) & (139.75, 161.5, 243.75) \\ 
 & & $2$ & mean (S.D.)&{148.836} (9.007)&{150.208} (13.575) &{153.194}  (28.251)&{160.026} (40.979)\\ 
& &  & quartiles & (144, 150, 152) & (142.75, 150, 154) & (138, 150, 158) & (137.75, 150, 165) \\  
 &  &  &  &  &  &  \\ 
& $0.5$ & $0.5$ & mean (S.D.) &{297.712} (43.291) &{302.204} (77.719)&{302.866} (96.511)&{297.662} (110.175)\\ 
& &  & quartiles  & (277, 297, 320) & (262, 300, 337) & (248, 298.5, 369.5) & (215, 301, 369.5) \\ 
& & $1$ & mean (S.D.)&{299.052} (16.132) &{299.910} (28.907)&{302.386} (55.267)&{300.956} (62.821)\\
& &  & quartiles& (290, 299, 308) & (288, 300, 313) & (277, 300, 324.25) & (270, 300, 329) \\ 
& & $2$ & mean (S.D.)&{300.010} (6.054)&{299.612} (10.079)&{298.844} (14.059)&{301.424} (21.022)\\ 
 & &  & quartiles  & (297, 300, 303.25) & (294, 300, 305) & (291, 300, 307) & (289, 300, 312) \\ 
  &  &  &  &  &  &  \\ 
   &  &  &  &  &  &  \\ 
Pareto$(3, 1)$ & $0.25$ & $0.5$ & mean (S.D.) &{151.562} (18.392) &{155.034} (32.505)&{165.260} (58.363)&{182.706} (83.268)\\ 
& &  & quartiles & (142, 150, 157) & (140, 150, 163) & (136, 150, 173) & (136.75, 150, 196.25) \\  
& & $1$ & mean (S.D.)&{150.206} (9.116)&{150.272} (15.405)&{152.824} (25.074)&{166.602} (58.982)\\ 
& &  & quartiles & (145, 150, 154) & (143, 150, 156) & (140, 150, 159.25) & (136, 150, 174.25) \\ 
& & $2$ & mean (S.D.)&{149.210} (6.201)&{149.934} (11.821)&{151.946} (21.426)&{156.836} (39.311)\\ 
 & &  & quartiles  & (146, 150, 152) & (143, 150, 153) & (140, 150, 156) & (136, 150, 160.25) \\ 
 &  &  &  &  &  &  \\ 
& $0.5$ & $0.5$ & mean (S.D.)&{300.524} (11.841) &{299.488} (21.317)&{299.664} (37.136)&{295.048} (55.000)\\ 
& &  & quartiles & (294, 300, 307) & (290, 300, 310) & (287, 300, 317) & (280.75, 300, 318) \\ 
& & $1$ & mean (S.D.)&{300.498} (6.600)&{300.560} (10.383)&{299.520} (18.862)&{297.766} (28.308)\\ 
& &  & quartiles  & (297, 300, 304) & (296, 300, 306) & (292, 300, 309.25) & (289, 300, 312.25) \\ 
& & $2$ & mean (S.D.)&{300.444} (4.411)&{300.234} (7.517)&{300.524} (11.122)&{298.840} (16.004)\\ 
& &  & quartiles  & (298, 300, 303) & (296, 300, 304) & (295.75, 300, 307) & (292, 300, 308) \\ 
\end{tabular}
\label{sampling_distribution_SN_Wilcoxon}
\end{table}
\end{landscape}

\begin{landscape}
\begin{table}
\caption{Descriptive statistics of the sampling distribution of $\hat{k}_{\text{C}, 0}$ for a change in the mean based on $500$ replications of fractional Gaussian noise and Pareto time series of length $600$ with Hurst parameter $H$ and a change in mean in $\tau$ of height $h$. }
\begin{tabular}{l l l l r r r r}
margins & $\tau$ & $h$ &  & $H= 0.6$ & $H = 0.7$ & $H = 0.8$ & $H = 0.9$ \\ 
\hline
normal & $0.25$ & $0.5$ & mean (S.D.)&{193.060} (64.917)& {228.948} (101.442)&{253.114} (111.182)&{271.380} (114.590)\\ 
& &  & quartiles& (150, 166.5, 222) & (151, 191.5, 286.75) & (156.75, 226, 341.5) & (172.75, 249.5, 354.25) \\ 
& & $1$ & mean (S.D.)&{162.028}  (22.948)&{173.838} (39.845)&{187.386} (63.865) &{213.114} (87.356)\\ 
& &  & quartiles & (150, 153, 164) & (150, 156.5, 187.25) & (150, 158, 206) & (150, 173, 254.25) \\ 
& & $2$ & mean (S.D.)&{152.374} (6.249) &{154.878} (10.395)&{159.700} (22.064) &{165.940} (33.124)\\ 
& &  & quartiles& (150, 150, 152) & (150, 150, 156) & (150, 151, 158) & (150, 150, 165) \\ 
 &  &  &  &  &  &  \\ 
 & $0.5$ & $0.5$ & mean (S.D.) &{297.840} (30.249)&{302.060} (63.878)&{300.246} (84.346)&{298.910} (97.904)\\ 
& &  & quartiles & (290, 299, 308) & (276, 301, 322) & (261.75, 300, 340) & (236.25, 299, 353.25) \\ 
& & $1$ & mean(S.D.) &{299.870} (9.356)&{299.662} (21.281)&{303.646} (42.245)&{299.762} (52.492)\\ 
& &  & quartiles & (298, 300, 302) & (297, 300, 304) & (293, 300, 307) & (290, 300, 311) \\ 
& & $2$ & mean (S.D.)&{300.060} (1.473)&{299.916} (3.199)&{299.442} (5.234)&{300.460} (8.179)\\ 
& &  & quartiles  & (300, 300, 300) & (300, 300, 300) & (300, 300, 300) & (300, 300, 300)\\ 
 &  &  &  &  &  &  \\ 
  &  &  &  &  &  &  \\ 
Pareto($3$, $1$) & $0.25$ & $0.5$ & mean (S.D.) &{175.632} (48.517)&{198.452} (79.303)&{205.506} (88.482)&{210.444}(93.831) \\ 
& &  & quartiles & (150, 159, 185) & (150, 168, 223.75) & (150, 173, 251.25) & (150, 167, 259.5) \\ 
& & $1$ & mean (S.D.)&{156.586} (14.133)&{160.350} (27.204)&{170.278} (45.402)&{177.278} (66.661)\\ 
& &  & quartiles & (150, 152, 159) & (150, 152, 161) & (150, 153, 171) & (150, 150, 174) \\ 
& & $2$ & mean (S.D.)&{150.314} (1.349)&{150.566} (3.984)&{152.474} (18.578)&{155.496} (29.408)\\ 
& &  & quartiles & (150, 150, 150) & (150, 150, 150) & (150, 150, 150) & (150, 150, 150) \\ 
 &  &  &  &  &  &  \\ 
 & $0.5$ & $0.5$ & mean (S.D.)&{296.260} (22.306) &{292.904} (43.471)&{289.192} (64.033)&{287.966} (64.827)\\  
& &  & quartiles  & (292, 300, 303.25) & (288.75, 300, 305) & (273.75, 300, 308.25) & (285, 300, 303) \\ 
& & $1$ & mean (S.D.)&{298.240} (6.104)&{297.306} (9.361)&{293.116} (26.614)&{292.864} (37.601)\\  
& &  & quartiles & (299, 300, 300) & (299, 300, 300) & (298, 300, 300) & (300, 300, 300) \\ 
& & $2$ & mean (S.D.)&{299.604} (1.843)&{299.228} (3.385)&{298.350} (8.354)&{297.632} (14.525)\\ 
& &  & quartiles  & (300, 300, 300) & (300, 300, 300) & (300, 300, 300) & (300, 300, 300)
\end{tabular}
\label{sampling_distribution_CUSUM}
\end{table}
\end{landscape}

\section{Proofs}\label{Proofs}
In the following let  $F_k$ and $F_{k+1, n}$ denote the empirical distribution  functions of the first $k$ and last $n-k$ realizations of $Y_1, \ldots, Y_n$,  i.e.
\begin{align*}
&F_k(x):=\frac{1}{k}\sum\limits_{i=1}^k1_{\left\{Y_i\leq x\right\}},\\
&F_{k+1, n}(x):=\frac{1}{n-k}\sum\limits_{i=k+1}^n1_{\left\{Y_i\leq x\right\}}.
\end{align*}
For notational convenience we write $W_n(k)$ instead of $W_{k, n}$ and $SW_{n}(k)$ instead of $SW_{k, n}$.
The proofs in this section as well as the proofs in the appendix are partially influenced by arguments that  have been established in \cite{HorvathKokoszka1997}, \cite{Bai1994} and \cite{DehlingRoochTaqqu2013a}.
In particular, some arguments are based on the empirical process 
non-central limit theorem
of
\cite{DehlingTaqqu1989} which states that
\begin{align*}
d_{n, r}^{-1}\lfloor n\lambda\rfloor (F_{\lfloor n\lambda\rfloor}(x)-F(x))
\overset{\mathcal{D}}{\longrightarrow}\frac{1}{r!}J_r(x)Z_H^{(r)}(\lambda), 
\end{align*}
where $r$ is the Hermite rank defined in Assumption \ref{ass:subordination}, $Z_H^{(r)}$ is an $r$-th order Hermite process\footnote{If $r=1$, the Hermite process 
 equals a standard fractional Brownian motion process with Hurst parameter $H=1-\frac{D}{2}$. We refer 
to \cite{Taqqu1979} for a general definition of  Hermite processes.},   $H=1-\frac{rD}{2}\in \left(\frac{1}{2}, 1\right)$, and
     ``$\overset{\mathcal{D}}{\longrightarrow}$'' denotes convergence in distribution with respect to the $\sigma$-field generated by the open balls in $D\left(\left[-\infty, \infty\right]\times \left[0, 1\right]\right)$, equipped with the supremum norm.
     
The Dudley-Wichura version of Skorohod's representation theorem (see \cite{ShorackWellner1986}, Theorem 2.3.4) implies that,
for our purposes, we may assume 
without loss of generality that
\begin{align*}
\sup\limits_{\lambda\in\left[0, 1\right],  x\in \mathbb{R}}\left|d_{n, r}^{-1}\lfloor n\lambda\rfloor\left(F_{\lfloor n\lambda\rfloor}(x)-F(x)\right)-\frac{1}{r!}J_r(x)Z_H^{(r)}(\lambda)\right|\longrightarrow 0
\end{align*}
almost surely.

\begin{proof}[Proof of Proposition \ref{Prop:consistency}]

The proof of Proposition \ref{Prop:consistency} is based on an application of Lemma \ref{Lem:W_process_under_A} in the appendix. According to Lemma  \ref{Lem:W_process_under_A} it holds that, under the assumptions of Proposition \ref{Prop:consistency},   
\begin{align*}
\frac{1}{n^2 h_n}\sum\limits_{i=1}^{\lfloor n\lambda\rfloor }\sum\limits_{j=\lfloor n\lambda\rfloor +1}^{n}\left(1_{\left\{X_i\leq X_j\right\}}-\frac{1}{2}\right)
\overset{P}{\longrightarrow} C\delta_{\tau}(\lambda), \ 0\leq \lambda \leq 1,
\end{align*}
where $\delta_{\tau}:[0, 1]\longrightarrow \mathbb{R}$  is defined by
\begin{align*}
\delta_{\tau}(\lambda)=
\begin{cases}
\lambda(1-\tau)  &\text{for} \ \lambda\leq \tau\\
(1-\lambda)\tau  &\text{for} \ \lambda\geq \tau
\end{cases}
\end{align*}
and $C$ denotes some non-zero constant.

It directly follows that  $\frac{1}{nd_{n, r}}\max_{1 \leq k\leq n-1}|W_{n}(k)|\overset{P}{\longrightarrow}\infty$.

Furthermore,
\begin{align*}
\frac{1}{n^2h_n}\max\limits_{1\leq k\leq \lfloor n(\tau-\varepsilon)\rfloor}\left|\sum\limits_{i=1}^k\sum\limits_{j=k+1}^n\left(1_{\left\{X_i\leq X_j\right\}}-\frac{1}{2}\right)\right|
\intertext{converges in probability to}
C\sup\limits_{0\leq \lambda\leq \tau-\varepsilon}\delta_{\tau}(\lambda)
=C(\tau-\varepsilon)(1-\tau)
\end{align*}
for any $0\leq \varepsilon <\tau$.

For  $\varepsilon>0$ define
\begin{align*}
Z_{n, \varepsilon}:=\frac{1}{n^2h_n}\max\limits_{1\leq k\leq \lfloor n\tau\rfloor}\left|W_n(k)\right|-\frac{1}{n^2h_n}\max\limits_{1\leq k\leq \lfloor n(\tau-\varepsilon)\rfloor}\left|W_n(k)\right|.
\end{align*}
As $Z_{n, \varepsilon}\overset{P}{\longrightarrow} C(1-\tau)\varepsilon$, it follows that $P(\hat{k}_W<\lfloor n(\tau-\varepsilon) \rfloor)=P(Z_{n, \varepsilon}= 0)\longrightarrow 0$.

An analogous line of argument yields 
\begin{align*}
P(\hat{k}_W>\lfloor n(\tau+\varepsilon)\rfloor)\longrightarrow 0.
\end{align*}
All in all, it follows that for any $\varepsilon >0$
\begin{align*}
&\lim\limits_{n\longrightarrow \infty}P\left(\left|\frac{\hat{k}_W}{n}-\tau\right|> \varepsilon \right)=0.
\end{align*}
This proves consistency of the change point estimator which is based on the  Wilcoxon test statistic.

In the following it is shown that $\frac{1}{n}\hat{k}_{SW}$ is a consistent estimator, too.
For this purpose, we consider the process $SW_n(\lfloor n\lambda\rfloor)$, $0\leq \lambda\leq 1$.
According to \cite{Betken2016}  the limit of the self-normalized Wilcoxon test statistic can be obtained by an application of the continuous mapping theorem to the process
\begin{align*}
\frac{1}{a_n} \sum\limits_{i=1}^{\lfloor n\lambda\rfloor}\sum\limits_{j=\lfloor n\lambda\rfloor+1}^n\left(1_{\left\{X_i\leq X_j\right\}}-\frac{1}{2}\right), \ 0\leq \lambda \leq 1, 
\end{align*}
where $a_n$ denotes an appropriate normalization.
Therefore, it follows by the corresponding argument  in \cite{Betken2016}   that 
\begin{align*}
SW_n(\lfloor n\lambda\rfloor)\overset{P}{\longrightarrow}\frac{\left|\delta_{\tau}(\lambda)\right|}{\left\{\int_0^{\lambda}\left(\delta_{\tau}(t)-\frac{t}{\lambda}\delta_{\tau}(\lambda)\right)^2dt+\int_{\lambda}^1\left(\delta_{\tau}(t)-\frac{1-t}{1-\lambda} \delta_{\tau}(\lambda)\right)^2dt \right\}^{\frac{1}{2}}}
\end{align*}
uniformly in $\lambda \in [0, 1]$.
Elementary calculations yield
\begin{align*}
&\sup\limits_{\lfloor n\tau_1\rfloor \leq k\leq k_0-n\varepsilon}SW_n(k)
\overset{P}{\longrightarrow}\sup\limits_{\tau_1\leq \lambda\leq \tau-\varepsilon}\frac{\sqrt{3}\lambda\sqrt{1-\lambda}}{(\tau-\lambda)},\\
&\sup\limits_{k_0+n\varepsilon \leq k\leq \lfloor n\tau_2\rfloor}SW_n(k)
\overset{P}{\longrightarrow}\sup\limits_{\tau +\varepsilon\leq \lambda\leq \tau_2}\frac{\sqrt{3}\sqrt{\lambda}(1-\lambda)}{(\tau-\lambda)}.
\end{align*}
As $SW_n(k_0)\overset{P}{\longrightarrow}\infty$ due to Theorem 2 in \cite{Betken2016}, we  conclude that $P(\hat{k}_{SW}>k_0+ n\varepsilon)$ and $P(\hat{k}_{SW}<k_0-n\varepsilon)$ converge to $0$ in probability. This proves $\frac{1}{n}\hat{k}_{SW}\overset{P}{\longrightarrow}\tau$.
\end{proof}

\begin{proof}[Proof of Theorem \ref{convergence rate}]
In the following we write $\hat{k}$ instead of $\hat{k}_W$.
For convenience, we assume that $h>0$ under fixed changes, and that for some $n_0\in \mathbb{N}$ $h_n>0$ for all $n\geq n_0$  under local changes, respectively. Furthermore, we  subsume both changes under the general assumption that $\lim_{n\rightarrow\infty}h_n=h$ (under fixed changes $h_n=h$ for all $n\in \mathbb{N}$, under local changes $h=0$).  
In order to prove Theorem \ref{convergence rate},  we need to show that for all $\varepsilon>0$ there exists an $n(\varepsilon)\in \mathbb{N}$ and an $M>0$ such that
\begin{align*}
P\left(\left|\hat{k}-k_0\right|>M m_n\right)<\varepsilon
\end{align*}
for all $n\geq n(\varepsilon)$.

For $M\in \mathbb{R}^{+}$ define $D_{n, M}:=\left\{k\in \left\{1, \ldots, n-1\right\}\left|\right.\left|k-k_0\right|>Mm_n\right\}$.

We have
\begin{align*}
P\left(\left|\hat{k}-k_0\right|>M m_n\right)
\leq P\left(\sup\limits_{k\in D_{n, M}}\left|W_n(k)\right|\geq |W_{n}(k_0)|\right)\leq P_1+P_2
\end{align*}
with
\begin{align*}
&P_1:=P\left(\sup\limits_{k\in  D_{n, M}}\left(W_n(k)- W_{n}(k_0)\right)\geq 0\right), \\
&P_2:=P\left(\sup\limits_{k\in  D_{n, M}}\left(-W_n(k)-W_{n}(k_0)\right)\geq 0\right).
\end{align*}

Note that $D_{n, M}=D_{n, M}(1)\cup D_{n, M}(2)$, where 
\begin{align*}
&D_{n, M}(1):=\left\{k\in \left\{1, \ldots, n-1\right\}\left|\right.k_0-k>Mm_n\right\}, \\
&D_{n, M}(2):=\left\{k\in  \left\{1, \ldots, n-1\right\}\left|\right.k-k_0>Mm_n\right\}.
\end{align*}
Therefore,  
$P_2\leq P_{2, 1}+P_{2, 2}$, where
\begin{align*}
&P_{2,1}:= P\left(\sup\limits_{k\in  D_{n, M}(1)}\left(-W_n(k)-W_{n}(k_0)\right)\geq 0\right),\\
&P_{2, 2}:=P\left(\sup\limits_{k\in  D_{n, M}(2)}\left(-W_n(k)-W_{n}(k_0)\right)\geq 0\right). 
\end{align*}
In the following we will consider the first summand only. (For the second summand analogous implications  result  from the same argument.)

For this, we define 
\begin{align*}
\widehat{W}_n(k):=\delta_n(k)\Delta(h_n),
\end{align*}
where 
\begin{align*}
\delta_n(k):=\begin{cases}
k(n-k_0), & k\leq k_0\\
k_0(n-k), & k> k_0
\end{cases}
\end{align*}
and
\begin{align*}
\Delta(h_n):=\int \left(F(x+h_n)-F(x)\right)dF(x).
\end{align*}

Note that
\begin{align*}
P_{2, 1}
&\leq P\left(\sup\limits_{k\in  D_{n, M}(1)}\left( \widehat{W}_n(k) -W_n(k)+\widehat{W}_n(k_0)-W_{n}(k_0)\right) \geq \widehat{W}_n(k_0)\right)\\
&\leq P\left(2\sup\limits_{\lambda \in \left[0, \tau\right]}\left|W_n(\lfloor n\lambda\rfloor)- \widehat{W}_n(\lfloor n\lambda\rfloor) \right| \geq k_0(n-k_0)\Delta(h_n)\right).
\end{align*}
We have
\begin{align*}
&\sup\limits_{\lambda \in \left[0, \tau\right]}\left|W_n(\lfloor n\lambda\rfloor)- \widehat{W}_n(\lfloor n\lambda\rfloor) \right|\\
&=\sup\limits_{\lambda \in \left[0, \tau\right]}\Biggl|\sum\limits_{i=1}^{\lfloor n\lambda\rfloor}\sum\limits_{j=\lfloor n\tau\rfloor+1}^n\left(1_{\left\{Y_i\leq Y_j+h_n\right\}}-\int F(x+h_n)dF(x)\right)\\
&\quad +\sum\limits_{i=1}^{\lfloor n\lambda\rfloor}\sum\limits_{j=\lfloor n\lambda\rfloor+1}^{\lfloor n\tau\rfloor}\left(1_{\left\{Y_i\leq Y_j\right\}}-\frac{1}{2}\right)\Biggr|.
\end{align*}
Due to  Lemma \ref{Lem} in the appendix and Theorem 1.1  in \cite{DehlingRoochTaqqu2013a}
\begin{align*}
2\sup_{\lambda \in \left[0, \tau\right]}\left|W_n(\lfloor n\lambda\rfloor)- \widehat{W}_n(\lfloor n\lambda\rfloor) \right|=\mathcal{O}_P\left(nd_{n, r}\right),
\end{align*}
 i.e. for all $\varepsilon >0$ there exists a $K>0$ such that
 \begin{align*}
 P\left(2\sup_{\lambda \in \left[0, \tau\right]}\left|W_n(\lfloor n\lambda\rfloor)- \widehat{W}_n(\lfloor n\lambda\rfloor) \right|\geq Knd_{n, r}\right)<\varepsilon
 \end{align*}
 for all $n$. 
Furthermore,  $k_0(n-k_0)\Delta(h_n)\sim Cn^2h_n$ for some constant $C$.
 Note that $Knd_{n, r}\leq k_0(n-k_0)\Delta(h_n)$ if and only if
 \begin{align*}
 K\leq\frac{k_0}{n}\frac{n-k_0}{n}\frac{\Delta(h_n)}{h_n}\frac{nh_n}{d_{n, r}}.
 \end{align*}
The right hand side of the above inequality diverges if $h_n=h$ is fixed or if  $h_n^{-1}=o\left(\frac{n}{d_{n, r}}\right)$. Therefore,  it is possible to find an $n(\varepsilon)\in \mathbb{N}$ such that
\begin{align*}
P_{2, 1}&\leq  P\left(2\sup\limits_{\lambda \in \left[0, \tau\right]}\left|W_n(\lfloor n\lambda\rfloor)- \widehat{W}_n(\lfloor n\lambda\rfloor) \right| \geq k_0(n-k_0)\Delta(h_n)\right)\\
&\leq P\left(2\sup\limits_{\lambda \in \left[0, \tau\right]}\left|W_n(\lfloor n\lambda\rfloor)- \widehat{W}_n(\lfloor n\lambda\rfloor) \right| \geq K nd_{n, r}\right)\\
&<\varepsilon
\end{align*}
for all  $n\geq n(\varepsilon)$.

We will now turn to the summand $P_1$. 
We have $P_{1}\leq P_{1,1}+P_{1,2}$, where
\begin{align*}
&P_{1, 1}:=P\left(\sup\limits_{k\in  D_{n, M}(1)}W_n(k)- W_{n}(k_0)\geq 0\right),\\ 
&P_{1, 2}:=P\left(\sup\limits_{k\in  D_{n, M}(2)}W_n(k)- W_{n}(k_0)\geq 0\right).
\end{align*}
In the following we will consider the first summand only. (For the second summand analogous implications  result  from the same argument.)

We define a random sequence $k_n$, $n \in \mathbb{N}$, by choosing $k_n\in  D_{n, M}(1)$ such that
\begin{align*}
&\sup\limits_{k\in  D_{n, M}(1)}\left(W_n(k)-\widehat{W}_n(k)+\widehat{W}_n(k_0)-W_{n}(k_0)\right)\\
&=W_n(k_n)-\widehat{W}_n(k_n)+\widehat{W}_n(k_0)-W_{n}(k_0).
\end{align*}
Note that for any  sequence $k_n$, $n\in\mathbb{N}$, with $k_n\in D_{n, M}(1)$
\begin{align*}
\widehat{W}_n(k_0)-\widehat{W}_n(k_n)
=(n-k_0) l_n\Delta(h_n)
\end{align*}
where $l_n:=k_0-k_n$.
Since $k_n\in D_{n, M}(1)$ and $m_n\longrightarrow \infty$ we have 
\begin{align*}
\frac{l_n}{d_{l_n, r}}=l_n^{1-H}L^{-\frac{r}{2}}(l_n)\geq (Mm_n)^{1-H}L^{-\frac{r}{2}}(Mm_n)
\end{align*}
for $n$ sufficiently large. Thus, we have
\begin{align*}
\frac{1}{nd_{l_n, r}}\left(\widehat{W}_n(k_0)-\widehat{W}_n(k_n)\right)&\geq\frac{n-k_0}{n}  \frac{m_n}{d_{m_n, r}}M^{1-H}\frac{L^{\frac{r}{2}}(m_n)}{L^{\frac{r}{2}}(Mm_n)}\Delta(h_n).
\end{align*}
If $h_n$ is fixed, the right hand side of the inequality diverges.
Under local changes the right hand side  asymptotically behaves like 
\begin{align*}
(1-\tau)M^{1-H}\int f^2(x)dx,
\end{align*}
since, in this case, $h_n\sim \frac{d_{m_n, r}}{m_n}$ due to the assumptions of Theorem \ref{convergence rate}.

In any case, for $\delta>0$ it is possible to find an $n_0\in \mathbb{N}$ such that 
\begin{align*}
\frac{1}{nd_{l_n, r}}\left(\widehat{W}_n(k_0)-\widehat{W}_n(k_n)\right)\geq   M^{1-H} (1-\tau)\int f^2(x)dx -\delta
\end{align*}
for all $n\geq n_0$.

All in all, the previous considerations show that there exists an $n_0\in \mathbb{N}$ and a constant $K$ such that for all $n\geq n_0$
\begin{align*}
P_{1,1}\leq P\left(\sup\limits_{k\in  D_{n, M}(1)}\frac{1}{nd_{k_0-k,r}}\left(W_n(k)-\widehat{W}_n(k)+\widehat{W}_n(k_0)-W_{n}(k_0)\right)  \geq  b(M)\right)
\end{align*}
where $b(M):=K M^{1-H}-\delta$ with $\delta>0$ fixed.

Some elementary calculations  show that for $k\leq k_0$
\begin{align*}
W_n(k)-\widehat{W}_n(k)+\widehat{W}_n(k_0)-W_n(k_0)
=A_{n, 1}(k)+A_{n, 2}(k)+A_{n, 3}(k)+A_{n, 4}(k),
\end{align*}
where
\begin{align*}
&A_{n, 1}(k):=-(n-k_0)(k_0-k)\int \left(F_{k, k_0}(x+h_n)-F(x+h_n)\right)d F_{k_0, n}(x),\\
&A_{n, 2}(k):=-(n-k_0)(k_0-k)\int \left( F_{k_0, n}(x)- F(x)\right)d F(x+h_n),\\
&A_{n, 3}(k):=(k_0-k)k\int \left(F_{k}(x)-F(x)\right)d F_{k, k_0}(x),\\
&A_{n, 4}(k):=-k (k_0-k)\int \left( F_{k, k_0}(x)- F(x)\right)dF(x).
\end{align*}
Thus, for $n\geq n_0$
\begin{align*}
P_{1, 1}
&\leq P\left(\sup\limits_{k\in  D_{n, M}(1)}\frac{1}{nd_{k_0-k, r}}\sum\limits_{i=1}^4\left|A_{n, i}(k)\right|\geq  b(M)\right)\\
&\leq \sum\limits_{i=1}^{4} P\left(\sup\limits_{k\in  D_{n, M}(1)}\frac{1}{nd_{k_0-k, r}}\left|A_{n, i}(k)\right|\geq \frac{1}{4}b(M)\right).
\end{align*}

For each $i\in \left\{1, \ldots, 4\right\}$ it will be shown that
\begin{align*}
P\left(\sup\limits_{k\in  D_{n, M}(1)}\frac{1}{nd_{k_0-k, r}}\left|A_{n, i}(k)\right|\geq \frac{1}{4}b(M)\right)<\frac{\varepsilon}{4} 
\end{align*}
for $n$ and $M$ sufficiently large.

\begin{enumerate}
\item 
Note that
\begin{align*}
&\sup\limits_{k\in  D_{n, M}(1)}\frac{1}{nd_{k_0-k, r}}\left|A_{n, 1}(k)\right|\\
&\leq \sup\limits_{k\in  D_{n, M}(1)}\sup\limits_{x \in \mathbb{R}}\left|d_{k_0-k, r}^{-1} (k_0-k)\left(F_{k, k_0}(x)-F(x)\right)\right|.
\end{align*}

Due to stationarity
\begin{align*}
&\sup\limits_{k\in  D_{n, M}(1)}\sup\limits_{x \in \mathbb{R}}\left|d_{k_0-k, r}^{-1} (k_0-k)\left(F_{k, k_0}(x)-F(x)\right)\right|\\
&\overset{\mathcal{D}}{=}\sup\limits_{k\in  D_{n, M}(1)}\sup\limits_{x \in \mathbb{R}}\left|d_{k_0-k, r}^{-1} (k_0-k)\left(F_{k_0-k}(x)-F(x)\right)\right|.
\end{align*}

Note that
\begin{align*}
&\sup\limits_{k\in  D_{n, M}(1)}\sup\limits_{x \in \mathbb{R}}\left|d_{k_0-k, r}^{-1} (k_0-k)\left(F_{k_0-k}(x)-F(x)\right)\right|\\
&\leq\sup\limits_{k\in  D_{n, M}(1)}\sup\limits_{x \in \mathbb{R}}\left|d_{k_0-k, r}^{-1} (k_0-k)\left(F_{k_0-k}(x)-F(x)\right)-\frac{1}{r!}Z_H^{(r)}(1)J_r(x)\right|\\
&\quad+\frac{1}{r!}\left|Z_H^{(r)}(1)\right|\sup\limits_{x\in \mathbb{R}}\left|J_r(x)\right|.
\end{align*}

Since  
\begin{align*}
\sup\limits_{x \in \mathbb{R}}\left|d_{n, r}^{-1} n\left(F_{n}(x)-F(x)\right)-\frac{1}{r!}Z_H^{(r)}(1)J_r(x)\right|\longrightarrow 0 \ a.s.
\end{align*} 
if $n\longrightarrow \infty$, 
and as  $k_0-k\geq M m_n$ with $m_n\longrightarrow \infty$, it follows that 
\begin{align*}
\sup\limits_{k\in  D_{n, M}(1)}\sup\limits_{x \in \mathbb{R}}\left|d_{k_0-k, r}^{-1} (k_0-k)\left(F_{k_0-k}(x)-F(x)\right)-\frac{1}{r!}Z_H^{(r)}(1)J_r(x)\right|
\end{align*}
converges to $0$ almost surely.
 Therefore, 
\begin{align*}
&P\left(\sup\limits_{k\in  D_{n, M}(1)}\frac{1}{nd_{k_0-k, r}}\left|A_{n, 1}(k)\right|\geq \frac{1}{4}b(M)\right)\\
&\leq P\left(\sup\limits_{k\in  D_{n, M}(1)}\sup\limits_{x \in \mathbb{R}}\left|d_{k_0-k, r}^{-1} (k_0-k)\left(F_{k, k_0}(x)-F(x)\right)\right|
\geq \frac{1}{4}b(M)\right)\\
&\leq P\left(\frac{1}{r!}\left|Z_H^{(r)}(1)\right|\sup\limits_{x\in \mathbb{R}}\left|J_r(x)\right|
\geq \frac{1}{4}b(M)\right)+\frac{\varepsilon}{8}.
\end{align*}
for $n$   sufficiently large.
Note that $\sup_{x\in \mathbb{R}}\left|J_r(x)\right|<\infty$.
Furthermore, it is well-known that all moments of Hermite processes are finite.
As a result, it follows by Markov's inequality that for some $M_0\in \mathbb{R}$
\begin{align*}
&P\left(\frac{1}{r!}\left|Z_H^{(r)}(1)\right|\sup\limits_{x\in \mathbb{R}}\left|J_r(x)\right|
\geq \frac{1}{4}b(M)\right)\leq \E\left|Z_H^{(r)}(1)\right|\frac{4 r!}{\sup\limits_{x\in \mathbb{R}}\left|J_r(x)\right|b(M)}<\frac{\varepsilon}{8}
\end{align*}
for all $M\geq M_0$.
\item 
We have
\begin{align*}
&\sup\limits_{k\in  D_{n, M}(1)}\frac{1}{nd_{k_0-k, r}}\left|A_{n, 2}(k)\right|\\
&\leq \left|d_{n, r}^{-1}(n-k_0)\int \left( F_{k_0, n}(x)- F(x)\right)d F(x+h_n)\right|
\end{align*}
for $n$ sufficiently large.
As a result,  
\begin{align*}
\sup\limits_{k\in  D_{n, M}(1)}\frac{1}{nd_{k_0-k, r}}\left|A_{n, 2}(k)\right|
\leq \sup\limits_{x\in\mathbb{R}}\left|d_{n, r}^{-1}(n-k_0) \left( F_{k_0, n}(x)- F(x)\right)\right|.
\end{align*}
Due to the empirical process non-central limit theorem of \cite{DehlingTaqqu1989} we have
\begin{align*}
\sup\limits_{x\in \mathbb{R}}\left|d_{n, r}^{-1}(n-k_0)\left( F_{k_0, n}(x)- F(x)\right)\right|\overset{\mathcal{D}}{\longrightarrow} \frac{1}{r!}\left|Z_H^{(r)}(1)-Z_H^{(r)}(\tau))\right|\sup\limits_{x\in \mathbb{R}} \left|J_r(x)\right|.
\end{align*}
Moreover,
\begin{align*}
\frac{1}{r!}\left|Z_H^{(r)}(1)-Z_H^{(r)}(\tau)\right|\sup\limits_{x\in \mathbb{R}} \left|J_r(x)\right|\overset{\mathcal{D}}{=} \frac{1}{r!} (1-\tau)^{H}\left|Z_H^{(r)}(1)\right|\sup\limits_{x\in \mathbb{R}}\left|J_r(x)\right|
\end{align*}
since $Z_H^{(r)}$ is a $H$-self-similar process with stationary increments.
Thus, we have
\begin{align*}
 &P\left(\sup\limits_{k\in  D_{n, M}(1)}\frac{1}{nd_{k_0-k, r}}\left|A_{n, 2}(k)\right|\geq \frac{1}{4}b(M)\right)\\
 &\leq
  P\left( \frac{1}{r!} (1-\tau)^{H}\left|Z_H^{(r)}(1)\right|\sup\limits_{x\in \mathbb{R}}\left|J_r(x)\right|\geq \frac{1}{4}b(M)\right)+\frac{\varepsilon}{8}
\end{align*}
for $n$ sufficiently large.
Again, it follows by Markov's inequality that
\begin{align*}
  P\left( \frac{1}{r!} (1-\tau)^{H}\left|Z_H^{(r)}(1)\right|\sup\limits_{x\in \mathbb{R}}\left|J_r(x)\right|\geq \frac{1}{4}b(M)\right)  <
\frac{\varepsilon}{8}
\end{align*}
for $M$ sufficiently large.
\item
Note that 
\begin{align*}
\frac{1}{n d_{k_0-k, r}}\left|A_{n, 3}(k)\right|
\leq\left|d_{n, r}^{-1}k\int \left(F_{k}(x)-F(x)\right) d F_{k, k_0}(x)\right|
\end{align*}
for $n$ sufficiently large.
Therefore, 
\begin{align*}
\sup\limits_{k\in  D_{n, M}(1)}\frac{1}{n d_{k_0-k, r}}\left|A_{n, 3}(k)\right|
\leq \sup\limits_{x\in \mathbb{R}, 0\leq \lambda\leq 1}\left|d_{n, r}^{-1}\lfloor n\lambda\rfloor  \left(F_{\lfloor n\lambda\rfloor}(x)-F(x)\right)\right|.
\end{align*}
The expression on the right hand side of  the inequality  converges in distribution to 
\begin{align*}
\frac{1}{r!}\sup\limits_{0\leq \lambda\leq 1}\left|Z_H^{(r)}(\lambda)\right|\sup\limits_{x\in \mathbb{R}}\left|J_r(x)\right|
\end{align*}
due to the empirical process non-central limit theorem.
Since 
\begin{align*}
\left\{Z_H^{(r)}(\lambda), \ 0\leq \lambda\leq 1\right\}
\overset{\mathcal{D}}{=}\left\{\lambda^H Z_H^{(r)}(1), \ 0\leq \lambda\leq 1\right\},
\end{align*}
we have 
\begin{align*}
\sup\limits_{0\leq \lambda\leq 1}\left|Z_H^{(r)}(\lambda)\right|
\overset{\mathcal{D}}{=}|Z_H^{(r)}(1)|.
\end{align*}
As a result, the aforementioned argument yields
\begin{align*}
 &P\left(\sup\limits_{k\in  D_{n, M}(1)}\frac{1}{nd_{k_0-k, r}}\left|A_{n, 3}(k)\right|\geq \frac{1}{4}b(M)\right)\\
 &\leq
  P\left(\frac{1}{r!}\left|Z_H^{(r)}(1)\right|\sup\limits_{x\in \mathbb{R}}\left|J_r(x)\right|\geq \frac{1}{4}b(M)\right)+\frac{\varepsilon}{8}\\
 &<
\frac{\varepsilon}{4}
\end{align*}
for $n$ and $M$ sufficiently large.
\item  
We have
\begin{align*}
&\sup\limits_{k\in  D_{n, M}(1)}\frac{1}{nd_{k_0-k, r}}\left|A_{n, 4}(k)\right|\\
&\leq \sup\limits_{k\in  D_{n, M}(1)}\sup\limits_{x\in \mathbb{R}}\left|d_{k_0-k, r}^{-1} (k_0-k)\left( F_{k, k_0}(x)- F(x)\right)\right|.
\end{align*}
Hence, the same argument that has been used to obtain an analogous result for $A_{n, 1}$  can be applied to  conclude that 
\begin{align*}
 &P\left(\sup\limits_{k\in  D_{n, M}(1)}\frac{1}{nd_{k_0-k, r}}\left|A_{n, 4}(k)\right|\geq \frac{1}{4}b(M)\right)<\frac{\varepsilon}{4}\\
\end{align*}
for $n$ and $M$ sufficiently large.
\end{enumerate}
All in all, it follows that for all $\varepsilon>0$ there exists an $n(\varepsilon)\in \mathbb{N}$ and an $M>0$ such that
\begin{align*}
P\left(\left|\hat{k}-k_0\right|>M m_n\right)<\varepsilon
\end{align*}
for all $n\geq n(\varepsilon)$. This proves Theorem \ref{convergence rate}.
\end{proof}
\begin{proof}[Proof of Theorem \ref{thm:asymp_distr}]
Note that 
\begin{align*}
&W_n^2(k_0+\lfloor m_n s\rfloor)-W_n^2(k_0)\\
&=\left(W_n(k_0+\lfloor m_n s\rfloor)-W_n(k_0)\right)\left(W_n(k_0+\lfloor m_n s\rfloor)+W_n(k_0)\right).
\end{align*}
We will show that (with an appropriate normalization) $W_n(k_0+\lfloor m_n s\rfloor)-W_n(k_0)$ converges in distribution to a non-deterministic limit process whereas $W_n(k_0+\lfloor m_n s\rfloor)+W_n(k_0)$ (with stronger normalization) converges in probability to a deterministic expression.
For notational convenience we write $d_{m_n}$ instead of $d_{m_n, 1}$, $J$ instead of $J_1$,  $\hat{k}$ instead of $\hat{k}_W$ and we define  $l_n(s):=k_0+\lfloor m_n s\rfloor$.
We have 
\begin{align*}
W_n(k_0+\lfloor m_n s\rfloor)-W_n(k_0)=\tilde{V}_n(l_n(s))+V_n(l_n(s)),
\end{align*}
where 
\begin{align*}
\tilde{V}_n(l)&=\begin{cases}
-\sum\limits_{i=l+1}^{k_0}\sum\limits_{j=k_0+1}^n\left(1_{\left\{Y_i\leq Y_j+h_n\right\}}-1_{\left\{Y_i\leq Y_j\right\}}\right)   &\text{if $s<0$}\\
- \sum\limits_{i=1}^{k_0}\sum\limits_{j=k_0+1}^{l}\left(1_{\left\{Y_i\leq Y_j+h_n\right\}}-1_{\left\{Y_i\leq Y_j\right\}}\right)  &\text{if $s>0$}\\
\end{cases}
\end{align*}
and
\begin{align*}
V_n(l)=\begin{cases}
\sum\limits_{i=1}^{l}\sum\limits_{j=l+1}^{k_0}\left(1_{\left\{Y_i\leq Y_j\right\}}-\frac{1}{2}\right)-\sum\limits_{i=l+1}^{k_0}\sum\limits_{j=k_0+1}^n\left(1_{\left\{Y_i\leq Y_j\right\}}-\frac{1}{2}\right)    &\text{if $s<0$}\\
\sum\limits_{i=k_0+1}^{l}\sum\limits_{j=l+1}^{n}\left(1_{\left\{Y_i\leq Y_j\right\}}-\frac{1}{2}\right)-\sum\limits_{i=1}^{k_0}\sum\limits_{j=k_0+1}^{l}\left(1_{\left\{Y_i\leq Y_j\right\}}-\frac{1}{2}\right)   &\text{if $s>0$}
\end{cases}.
\end{align*}

We will show that $\frac{1}{nd_{m_n}}\tilde{V}_n(l_n(s))$ converges to $h(s; \tau)$ in probability
and that  
$\frac{1}{nd_{m_n}}V_n(l_n(s))$ converges in distribution to $\sign(s)B_H(s)\int J(x)dF(x)$ in $D\left[-M, M\right]$.

We rewrite $\tilde{V}_n(l_n(s))$  in the following way:
\begin{align*}
&\tilde{V}_n(l_n(s))
=-(k_0-l_n(s))(n-k_0) \int \left(F_{l_n(s), k_0}(x+h_n)-F_{l_n(s), k_0}(x)\right)dF_{k_0, n}(x)
\intertext{if $s<0$,}
&\tilde{V}_n(l_n(s))
=- k_0(l_n(s)-k_0) \int \left(F_{k_0}(x+h_n)-F_{k_0}(x)\right)dF_{k_0, l_n(s)}(x) 
\end{align*}
if $s>0$.

For $s<0$ the limit of $\frac{1}{nd_{m_n}}\tilde{V}_n(l_n(s))$ corresponds to the limit of 
\begin{align*}
-(1-\tau)d_{m_n}^{-1}(k_0-l_n(s)) \int \left(F(x+h_n)-F(x)\right)dF(x) 
\end{align*}
due to Lemma \ref{Lem:int_sq_dens} and stationarity of the random sequence $Y_i$, $i\geq 1$.
 Note that
\begin{align*}
&d_{m_n}^{-1}(k_0-l_n(s)) \int \left(F(x+h_n)-F(x)\right)dF(x)\\
&=-d_{m_n}^{-1}\lfloor m_ns\rfloor h_n \int \frac{1}{h_n}\left(F(x+h_n)-F(x)\right)dF(x).
\end{align*}
The above expression converges to $-s\int f^2(x)dx$, since $h_n\sim \frac{d_{m_n}}{m_n}$.

\vspace{2mm}

For $s>0$
the limit of $\frac{1}{nd_{m_n}}\tilde{V}_n(l_n(s))$ corresponds to the limit of 
\begin{align*}
- \tau d_{m_n}^{-1}(l_n(s)-k_0) \int \left(F(x+h_n)-F(x)\right)dF(x)
\end{align*}
due to Lemma \ref{Lem:int_sq_dens} and stationarity of the random sequence $Y_i$, $i\geq 1$.
Note that
\begin{align*}
&d_{m_n}^{-1}(l_n(s)-k_0) \int \left(F(x+h_n)-F(x)\right)dF(x)\\
&=d_{m_n}^{-1}\lfloor m_ns\rfloor h_n\int \frac{1}{h_n}\left(F(x+h_n)-F(x)\right)dF(x)
\end{align*}
The above expression converges to $s\int f^2(x)dx$, since $h_n\sim \frac{d_{m_n}}{m_n}$.\\
All in all,  it follows that $\frac{1}{nd_{m_n}}\tilde{V}_n(l_n(s))$ converges to 
$h(s; \tau)$ defined by
\begin{align*}
h(s; \tau)=
\begin{cases}
s(1-\tau)\int f^2(x)dx  &\text{if $s\leq 0$}\\
-s\tau \int f^2(x)dx  &\text{if $s> 0$}
\end{cases}
.
\end{align*}

\vspace{5mm}

In the following it is shown that $\frac{1}{nd_{m_n}}V_n(l_n(s))$ converges in distribution to
\begin{align*}
\sign(s)B_H(s)\int J(x)dF(x), \ -M\leq s\leq M.
\end{align*}
Note that if $s<0$,
\begin{align*}
V_n(l_n(s))
=&-l_n(s)(k_0-l_n(s))\int 
\left(F_{l_n(s), k_0}(x)-F(x)\right)dF_{l_n(s)}(x)\\
&-(k_0-l_n(s))(n-k_0)\int \left(F_{l_n(s), k_0}(x)-F(x)\right)dF_{k_0, n}(x)\\
&+l_n(s)(k_0-l_n(s))\int (F_{l_n(s)}(x)-F(x))dF(x)\\
&+(k_0-l_n(s))(n-k_0)\int \left(F_{k_0, n}(x)-F(x)\right)dF(x).
\end{align*}
If $s>0$, we have
\begin{align*}
V_n(l_n(s))
=&(l_n(s)-k_0)(n-l_n(s))\int \left(F_{k_0, l_n(s)}(x)-F(x)\right)dF_{l_n(s), n}(x)\\
&+k_0(l_n(s)-k_0)\int\left(F_{k_0, l_n(s)}(x)-F(x)\right)(x)dF_{k_0}(x)\\
&-(l_n(s)-k_0)(n-l_n(s))\int \left(F_{l_n(s), n}(x)-F(x)\right)dF(x)\\
&-k_0(l_n(s)-k_0)\int \left(F_{k_0}(x)- F(x)\right)dF(x).
\end{align*}
The arguments that appear in the proof of Lemma \ref{Lem:int_sq_dens} can also be applied to show that the limit of $\frac{1}{nd_{m_n}}V_n(l_n(s))$ corresponds to the limit of
\begin{align*}
\frac{1}{nd_{m_n}}\left(A_{1, n}(s)+A_{2, n}(s)+A_{3, n}(s)\right),
\end{align*}
where
\begin{align*}
&A_{1, n}(s):=
 (-l_n(s)-n+k_0)(k_0-l_n(s))\int 
\left(F_{l_n(s), k_0}(x)-F(x)\right)dF(x)
\intertext{if $s<0$,}
&A_{1, n}(s):=(n-l_n(s)+k_0)(l_n(s)-k_0)\int 
\left(F_{k_0, l_n(s)}(x)-F(x)\right)dF(x) 
\intertext{if $s>0$,}
&A_{2, n}(s):=
\begin{cases}
(k_0-l_n(s))l_n(s)\int (F_{l_n(s)}(x)-F(x))dF(x)  &\text{if $s<0$}\\
-(l_n(s)-k_0)(n-l_n(s))\int\left(F_{l_n(s), n}(x)-F(x)\right)dF(x)  &\text{if $s>0$}
\end{cases},\\
&A_{3, n}(s):=
\begin{cases}
(k_0-l_n(s))(n-k_0)\int \left(F_{k_0, n}(x)-F(x)\right)dF(x)  &\text{if $s<0$}\\
-(l_n(s)-k_0)k_0\int \left(F_{k_0}(x)-F(x)\right)dF(x)  &\text{if $s>0$}
\end{cases}.
\end{align*}

Note that for $s<0$
\begin{align*} 
 \frac{1}{nd_{m_n}}A_{2, n}(s)=-\frac{1}{nd_{m_n}}\lfloor m_ns\rfloor l_n(s)\int (F_{l_n(s)}(x)-F(x))dF(x).
 \end{align*} 
The above expression	  converges to $0$ uniformly in $s$, since $\frac{m_n}{d_{m_n}}=o(\frac{n}{d_n})$ and since
\begin{align*}
&\sup\limits_{-M\leq s\leq 0}\left|d_n^{-1}l_n(s)\int (F_{l_n(s)}(x)-F(x))dF(x)\right|\\
&\leq \sup\limits_{x, \lambda}\left|
d_n^{-1}\lfloor n\lambda\rfloor(F_{\lfloor n\lambda\rfloor}(x)-F(x))-
B_H(\lambda) J(x)\right|\\
& \quad+\sup\limits_{0\leq\lambda\leq1}\left|
B_H(\lambda)\right| \left|\int J(x)dF(x)\right|,
\end{align*}
i.e. $\sup_{-M\leq s\leq 0}\left|d_n^{-1}l_n(s)\int (F_{l_n(s)}(x)-F(x))dF(x)\right|$ is bounded in probability.
An analogous argument shows that $\frac{1}{nd_{m_n}}A_{3, n}(s)$ vanishes if $n$ tends to $\infty$.

Therefore, it remains to show that  $\frac{1}{nd_{m_n}}A_{1, n}(s)$  converges in distribution to a non-deterministic expression.
Due to stationarity 
\begin{align*}
 \frac{1}{nd_{m_n}}A_{1, n}(s)\overset{\mathcal{D}}{=} \frac{n+\lfloor m_ns\rfloor}{n}d_{m_n}^{-1}\lfloor m_ns\rfloor\int 
\left(F_{-\lfloor m_ns\rfloor}(x)-F(x)\right)dF(x)
\end{align*}
for $s<0$.
As a result,  $ \frac{1}{nd_{m_n}}A_{1, n}(s)$  converges in distribution to $-B_H(s)\int J(x)dF(x)$.

If  $s>0$, an application of the previous arguments shows that  $\frac{1}{nd_{m_n}}A_{2, n}(s)$ and $\frac{1}{nd_{m_n}}A_{3, n}(s)$ converge to $0$ whereas $\frac{1}{nd_{m_n}}A_{1, n}(s)$  converges in distribution to $B_H(s)\int J(x)dF(x)$.

All in all, it follows that
\begin{align*}
\frac{1}{nd_{m_n}}\left(W_n(k_0+\lfloor m_n s\rfloor)-W_n(k_0)\right)\overset{\mathcal{D}}{\longrightarrow}
\sign(s)B_H(s)\int J(x)dF(x)+h(s; \tau)
\end{align*}
in $D[-M, M]$.

Furthermore, it follows that with the stronger normalization $h_nn^2$ the limit of $\frac{1}{h_n n^2}W_n(k_0+\lfloor m_n s\rfloor)$
corresponds to the limit of $\frac{1}{h_n n^2}W_n(k_0)$.

We have
\begin{align*}
\frac{1}{h_nn^2}W_n(k_0)
&=\frac{1}{h_nn^2}k_0(n-k_0)\int\left(F_{k_0}(x+h_n)-F_{k_0}(x)\right)dF_{k_0, n}(x)\\
&\quad +\frac{1}{h_nn^2}\sum\limits_{i=1}^{k_0}\sum\limits_{j=k_0+1}^n\left(1_{\left\{Y_i\leq Y_j\right\}}-\frac{1}{2}\right).
\end{align*}

The second summand on the right hand side vanishes as $n$ tends to $\infty$, since $h_n^{-1}=o\left(n/d_n\right)$.
Due to Lemma \ref{Lem:int_sq_dens} the limit of  $d_n^{-1}k_0\int\left(F_{k_0}(x+h_n)-F_{k_0}(x)\right)dF_{k_0, n}(x)$ corresponds to the limit of  $d_n^{-1}k_0\int\left(F(x+h_n)-F(x)\right)dF(x)$. Therefore, 
\begin{align*}
h_n^{-1}\int\left(F_{k_0}(x+h_n)-F_{k_0}(x)\right)dF_{k_0, n}(x)\longrightarrow \int f^2(x)dx \quad a.s.
\end{align*}
In addition, $\frac{k_0}{n}\frac{(n-k_0)}{n}\longrightarrow \tau(1-\tau)$.

From this we can conclude that
\begin{align*}
\frac{1}{h_nn^2}\left(W_n(k_0+m_ns)+W_n(k_0)\right)\overset{P}{\longrightarrow}2\tau(1-\tau)\int f^2(x)dx
\end{align*}
in $D[-M, M]$.
This completes the proof of the first assertion in Theorem \ref{thm:asymp_distr}.

In order to show  that
\begin{align*}
&m_n^{-1}(\hat{k}-k_0)\overset{\mathcal{D}}{\longrightarrow}\argmax_{-\infty < s <\infty}\left(\sign(s)B_H(s)\int J(x)dF(x)+h(s; \tau)\right),
\end{align*}
we make use of Lemma \ref{Lem:sargmax}.

For this purpose, we note that according to Lifshits' criterion for unimodality of Gaussian processes  (see Theorem 1.1  in \cite{Ferger1999})   the random function $G_{H, \tau}(s)=\sign(s)B_H(s)\int J(x)dF(x)+h(s; \tau)$ attains its maximal value in $[-M, M]$ at a unique point  with probability $1$ for every $M>0$.
Hence,  an application of Lemma \ref{Lem:sargmax} in the appendix yields
\begin{align*}
\sargmax_{s\in [-M, M]}\frac{1}{e_n}\left(W_n^2(k_0+\lfloor m_n s\rfloor)-W_n^2(k_0)\right)\overset{\mathcal{D}}{\longrightarrow}\argmax_{s\in [-M, M]}G_{H, \tau}(s).
\end{align*}
It remains to be shown that instead of considering the $\sargmax$ in $[-M, M]$ we may as well consider the smallest $\argmax$ in $\mathbb{R}$.
By the law of the iterated logarithm for fractional Brownian motions we have  $\lim_{|s|\rightarrow\infty}\frac{B_H(s)}{s}=0$ a.s. so that
$\sign(s)B_H(s)\int J(x)dF(x)+h(s; \tau)\longrightarrow -\infty$  a.s. if $|s|\rightarrow \infty$. Therefore,
the limit corresponds to $\argmax_{s\in (-\infty, \infty)}G_{H, \tau}(s)$ if  $M$ is sufficiently large.

For $M>0$ define 
\begin{align*}
\hat{\hat{k}}(M):=\min\left\{k: \left|k_0- k\right|\leq  Mm_n, \ \left|W_{n}(k)\right|=\max\limits_{|k_0-i|\leq Mm_n}\left|W_{n}(i)\right|\right\}.
\end{align*}

Note that 
\begin{align*}
&\Bigl|\sargmax_{s\in [-M, M]}\left(W_n^2(k_0+\lfloor m_n s\rfloor)-W_n^2(k_0)\right)\\
&-\sargmax_{s\in (-\infty, \infty)}\left(W_n^2(k_0+\lfloor m_n s\rfloor)-W_n^2(k_0)\right)\Bigr|\\
&=m_n^{-1}\left|\hat{\hat{k}}(M)-\hat{k}\right| + \mathcal{O}_P(1).
\end{align*}
Therefore, we have to show  that for some $M\in \mathbb{R}$ 
\begin{align*}
m_n^{-1}\left|\hat{\hat{k}}(M)-\hat{k}\right|
\overset{P}{\longrightarrow}0
\end{align*}
as $n$ tends to infinity.
Note that
\begin{align*}
P\left(\hat{k}=\hat{\hat{k}}(M)\right)
&= P\left(\left|\hat{k}-k_0\right|\leq Mm_n\right)\\
&=1-P\left(\left|\hat{k}-k_0\right|> Mm_n\right).
\end{align*}
Furthermore, we have
\begin{align*}
&\lim\limits_{M\rightarrow\infty}\liminf_{n\rightarrow \infty}\left(1-P\left(|\hat{k}-k_0|>Mm_n\right)\right)\\
&=1-\lim\limits_{M\rightarrow\infty}\limsup_{n\rightarrow \infty}P\left(|\hat{k}-k_0|>Mm_n\right)\\
&=1
\end{align*}
because $|\hat{k}-k_0|=O_P(m_n)$ by Theorem \ref{convergence rate}.
As a result, we have
\begin{align*}
\lim\limits_{M\rightarrow\infty}\liminf_{n\rightarrow \infty}P\left(\hat{k}=\hat{\hat{k}}(M)\right)=1.
\end{align*}
Hence,  for all $\varepsilon>0$ there is an $M_0\in \mathbb{R}$ and an $n_0\in \mathbb{N}$ such that 
\begin{align*}
P\left(\hat{k}\neq \hat{\hat{k}}(M)\right)<\varepsilon
\end{align*}
for all $n\geq n_0$ and all $M\geq M_0$.
This concludes the proof of Theorem \ref{thm:asymp_distr}.
\end{proof}

\newpage
\bibliographystyle{imsart-nameyear} 
\bibliography{PaperCPE}

\appendix

\section{Auxiliary Results}

In the following we prove some Lemmas that are needed for the proofs of our main results.
Lemma \ref{Lem:W_process_under_A} characterizes the asymptotic behavior of the Wilcoxon process under the assumption of a change-point in the mean. It is used to prove consistency of the change-point estimators $\hat{k}_{\text{W}}$ and $\hat{k}_{SW}$. 
\begin{Lem}\label{Lem:W_process_under_A}
Define $\delta_{\tau}:[0, 1]\longrightarrow \mathbb{R}$   by
\begin{align*}
\delta_{\tau}(\lambda)=
\begin{cases}
\lambda(1-\tau)  &\text{for} \ \lambda\leq \tau\\
(1-\lambda)\tau  &\text{for} \ \lambda\geq \tau
\end{cases}.
\end{align*}
Assume that Assumption \ref{ass:subordination} holds and that either 
\begin{enumerate}
\item[a)]  $h_n =  h$ with $h\neq 0$, 
\end{enumerate}
or
\begin{enumerate}
\item[b)] $\lim_{n\rightarrow \infty}h_n=0$ with $h_n^{-1}=\hbox{o}\left(\frac{n}{d_{n, r}}\right)$ and   $F$ has a bounded
density $f$.
\end{enumerate}
Then, we have
\begin{align*}
\frac{1}{n^2h_n}\sum\limits_{i=1}^{\lfloor n\lambda\rfloor}\sum\limits_{j=\lfloor n\lambda\rfloor+1}^n\left(1_{\left\{X_i\leq X_j\right\}}-\frac{1}{2}\right)\overset{P}{\longrightarrow} C\delta_{\tau}(\lambda), \ 0\leq \lambda \leq 1,
\end{align*}
where
\begin{align*}
C:=\begin{cases}
\frac{1}{h}\int \left(F(x+h)-F(x)\right)dF(x)    &\text{if } h_n =  h, \ h\neq 0,\\
\int f^2(x)dx  &\text{if } \lim_{n\rightarrow \infty}h_n= 0 \text{ and } h_n^{-1}=\hbox{o}\left(\frac{n}{d_{n, r}}\right)
\end{cases}.\notag
\end{align*}
\end{Lem}
 
\begin{proof}
First, consider the case $h_n=h$ with $h\neq 0$.
For $\lfloor n\lambda\rfloor\leq \lfloor n\tau\rfloor$
we have
\begin{align*}
&\frac{1}{n^2}\sum \limits_{i=1}^{\lfloor n\lambda\rfloor}\sum \limits_{j=\lfloor n\lambda\rfloor+1}^n\left(1_{\left\{X_i\leq X_j\right\}}-\frac{1}{2}\right)\\
&=\frac{1}{n^2}\sum \limits_{i=1}^{\lfloor n\lambda\rfloor}\sum \limits_{j=\lfloor n\tau\rfloor+1}^n\left(1_{\left\{Y_i\leq Y_j+h\right\}}-\frac{1}{2}\right)
 +\frac{1}{n^2}\sum \limits_{i=1}^{\lfloor n\lambda\rfloor}\sum \limits_{j=\lfloor n\lambda\rfloor+1}^{\lfloor n\tau\rfloor}\left(1_{\left\{Y_i\leq Y_j\right\}}-\frac{1}{2}\right).
\end{align*}
By Lemma 1 in  \cite{Betken2016} the first summand on the right hand side of the equation converges in probability to $\lambda(1-\tau)\int\left(F(x+h)-F(x)\right)dF(x)$ uniformly in $\lambda\leq \tau$. The second summand vanishes as $n$ tends to $\infty$.

If $\lfloor n\lambda\rfloor> \lfloor n\tau\rfloor$, 
\begin{align*}
&\frac{1}{n^2}\sum \limits_{i=1}^{\lfloor n\lambda\rfloor}\sum \limits_{j=\lfloor n\lambda\rfloor+1}^n\left(1_{\left\{X_i\leq X_j\right\}}-\frac{1}{2}\right)\\
&=\frac{1}{n^2}\sum \limits_{i=1}^{\lfloor n\tau\rfloor}\sum \limits_{j=\lfloor n\lambda\rfloor+1}^n\left(1_{\left\{Y_i\leq Y_j+h\right\}}-\frac{1}{2}\right)+\frac{1}{n^2}\sum \limits_{i=\lfloor n\tau\rfloor +1}^{\lfloor n\lambda\rfloor}\sum \limits_{j=\lfloor n\lambda\rfloor+1}^{n}\left(1_{\left\{Y_i\leq Y_j\right\}}-\frac{1}{2}\right).
\end{align*}
In this case,
the first summand on the right hand side of the equation converges in probability to $(1-\lambda)\tau\int\left(F(x+h)-F(x)\right)dF(x)$ uniformly in $\lambda\geq\tau$   due to Lemma 1 in \cite{Betken2016} while the second summand  converges  in probability to zero.
All in all, it follows that
\begin{align*}
\frac{1}{n^2}\sum\limits_{i=1}^{\lfloor n\lambda\rfloor }\sum\limits_{j=\lfloor n\lambda\rfloor +1}^{n}\left(1_{\left\{X_i\leq X_j\right\}}-\frac{1}{2}\right)
\overset{P}{\longrightarrow}\delta_{\tau}(\lambda)\int\left(F(x+h)-F(x)\right)dF(x)
\end{align*}
uniformly in $\lambda\in[0,1]$.

If $\lim_{n\rightarrow \infty}h_n= 0$, the process
\begin{align*}
&\frac{1}{nd_{n, r}}\sum\limits_{i=1}^{\lfloor n\lambda\rfloor}\sum\limits_{j=\lfloor n\lambda\rfloor+1}^n\left(1_{\left\{X_i\leq X_j\right\}}-\frac{1}{2}\right)\\
&-\frac{n}{d_{n, r}}\delta_{\tau}(\lambda)\int\left(F(x+h_n)-F(x)\right)dF(x),  \ 0\leq \lambda \leq 1,
\end{align*}
converges in distribution to
\begin{align*}
\frac{1}{r!}\int J_r(x)dF(x)\left(Z_H^{(r)}(\lambda)-\lambda Z_H^{(r)}(1)\right), \ 0\leq\lambda \leq 1,
\end{align*}
due to Theorem 3.1 in \cite{DehlingRoochTaqqu2013b}.
By assumption $h_n^{-1}=o\left(\frac{n}{d_{n, r}}\right)$, so that
\begin{align*}
\frac{1}{n^2h_n}\sum\limits_{i=1}^{\lfloor n\lambda\rfloor}\sum\limits_{j=\lfloor n\lambda\rfloor+1}^n\left(1_{\left\{X_i\leq X_j\right\}}-\frac{1}{2}\right)\overset{P}{\longrightarrow}  \delta_{\tau}(\lambda)\int f^2(x)dx, \ 0\leq\lambda \leq 1.
\end{align*}
\end{proof}

The proof of Theorem \ref{convergence rate}, which establishes a convergence rate for the estimator $\hat{k}_{\text{W}}$, requires the following  result:

\begin{Lem}\label{Lem}
Suppose that Assumption \ref{ass:subordination} holds and let  $h_n$, $n\in \mathbb{N}$,  be a sequence of real numbers with $\lim_{n\rightarrow \infty}h_n=h$.
\begin{enumerate}
\item 
The process
\begin{align*}
\frac{1}{nd_{n, r}}\sum\limits_{i=1}^{\lfloor n\lambda\rfloor}\sum\limits_{j=\lfloor n\tau\rfloor+1}^n\left(1_{\left\{Y_i\leq Y_j+h_n\right\}}-\int F(x+h_n)dF(x)\right), \ 0\leq\lambda \leq \tau,
\end{align*}
converges in distribution to
\begin{align*}
&(1-\tau)\frac{1}{r!}Z_H^{(r)}(\lambda)\int J_r(x+h)dF(x)\\
&-\lambda\frac{1}{r!}\left(Z_H^{(r)}(1)-Z_H^{(r)}(\tau)\right)\int J_r(x)dF(x+h)
\end{align*}
uniformly in $\lambda\leq \tau$.
\item
The process
\begin{align*}
\frac{1}{nd_{n, r}}\sum\limits_{i=1}^{\lfloor n\tau\rfloor}\sum\limits_{j=\lfloor n\lambda\rfloor+1}^n\left(1_{\left\{Y_i\leq Y_j+h_n\right\}}-\int F(x+h_n)dF(x)\right), \ \tau\leq\lambda \leq 1,
\end{align*}
converges in distribution to
\begin{align*}
&(1-\lambda)\frac{1}{r!}Z_H^{(r)}(\tau)\int J_r(x+h)dF(x)\\
&-\tau\frac{1}{r!}\left(Z_H^{(r)}(1)-Z_H^{(r)}(\lambda)\right)\int J_r(x)dF(x+h)
\end{align*}
uniformly in $\lambda\geq \tau$.
\end{enumerate}
\end{Lem}

\begin{proof}
We give a proof for the first assertion only as the convergence of the second term follows by an analogous argument.
The steps in  this proof correspond  to the argument that proves Theorem 1.1 in  \cite{DehlingRoochTaqqu2013a}.

For $\lambda\leq \tau$ it follows that
\begin{align*}
\sum\limits_{i=1}^{\lfloor n\lambda\rfloor}\sum\limits_{j=\lfloor n\tau\rfloor+1}^n1_{\left\{Y_i\leq Y_j+h_n\right\}}
&=\left(n-\lfloor n\tau\rfloor\right)\lfloor n\lambda\rfloor
\int F_{\lfloor n\lambda\rfloor}( x+h_n)
dF_{\lfloor n \tau \rfloor +1, n}(x).
\end{align*}
This yields the following decomposition:
\begin{align}\label{decomposition}
&\frac{1}{nd_{n, r}}\sum\limits_{i=1}^{\lfloor n\lambda\rfloor}\sum\limits_{j=\lfloor n\tau\rfloor+1}^n\left(1_{\left\{Y_i\leq Y_j+h_n\right\}}-\int F(x+h_n)dF(x)\right)\\
&=\frac{n-\lfloor n\tau\rfloor}{n}d_{n, r}^{-1}\lfloor n\lambda\rfloor
\int\left( F_{\lfloor n\lambda\rfloor}( x+h_n)
-F(x+h_n)\right)dF_{\lfloor n \tau \rfloor +1, n}(x)\notag\\
&\quad +\frac{n-\lfloor n\tau\rfloor}{n}d_{n, r}^{-1}\lfloor n\lambda\rfloor
\int F(x+h_n)
d\left(F_{\lfloor n \tau \rfloor +1, n}-F\right)(x).\notag
\end{align}

For the first summand we have
\begin{align*}
&\sup\limits_{0\leq\lambda\leq \tau}\Bigl|d_{n, r}^{-1}\lfloor n\lambda\rfloor
\int \left(F_{\lfloor n\lambda\rfloor}( x+h_n)
-F(x+h_n)\right)dF_{\lfloor n \tau \rfloor +1, n}(x)\\
&\qquad \ \ \ -\frac{1}{r!}Z_H^{(r)}(\lambda)\int J_r(x+h)dF(x)\Bigr|\\
&\leq\sup\limits_{0\leq\lambda\leq \tau}\Bigl|\int d_{n, r}^{-1}\lfloor n\lambda\rfloor\left(F_{\lfloor n\lambda\rfloor}( x+h_n)
-F(x+h_n)\right)\\
&\qquad\quad \ \ \ -\frac{1}{r!}Z_H^{(r)}(\lambda)J_r(x+h_n)dF_{\lfloor n \tau \rfloor +1, n}(x)\Bigr|\\
&\quad +\frac{1}{r!}\sup\limits_{0\leq\lambda\leq \tau}\left|Z_H^{(r)}(\lambda)\right|\left|\int\left(J_r(x+h_n)-J_r(x+h)\right)dF_{\lfloor n \tau \rfloor +1, n}(x)\right|\\
&\quad +\frac{1}{r!}\sup\limits_{0\leq\lambda\leq \tau}\left|Z_H^{(r)}(\lambda)\right|\left|\int J_r(x+h)d\left(F_{\lfloor n \tau \rfloor +1, n}-F\right)(x)\right|.
\end{align*}

We will show 
that each of the summands on the right hand side converges to $0$. 
The first summand converges to $0$ because
of the empirical non-central limit theorem of \cite{DehlingTaqqu1989}.
In order to show convergence of the second and third summand, note that $\sup_{0\leq\lambda\leq \tau}|Z_H^{(r)}(\lambda)|<\infty$ a.s. since the sample paths of the Hermite processes  are almost surely continuous.
 
Furthermore, we have 
\begin{align*}
\int J_r(x+h)dF_{\lfloor n \tau \rfloor +1, n}(x)
&=-\int \int 1_{\left\{x+h\leq G(y)\right\}}H_r(y)\varphi(y)dy dF_{\lfloor n \tau \rfloor +1, n}(x)\\
&=-\int \int 1_{\left\{x\leq G(y)-h\right\}}dF_{\lfloor n \tau \rfloor +1, n}(x)H_r(y)\varphi(y)dy \\
&=-\int F_{\lfloor n \tau \rfloor +1, n}(G(y)-h) H_r(y)\varphi(y)dy.
\end{align*}
Analogously, it follows that
\begin{align*}
\int J_r(x+h_n)dF_{\lfloor n \tau \rfloor +1, n}(x)=-\int F_{\lfloor n \tau \rfloor +1, n}(G(y)-h_n) H_r(y)\varphi(y)dy.
\end{align*}
Therefore, we may conclude that
\begin{align*}
&\left|\int \left(J_r(x+h_n)-J_r(x+h)\right)dF_{\lfloor n \tau \rfloor +1, n}(x)\right|\\
&\leq  2\sup\limits_{x\in\mathbb{R}}\left|F_{\lfloor n \tau \rfloor +1, n}(x)-F(x)\right|\int \left|H_r(y)\right|\varphi(y)dy\\
&\quad + \int \left|F(G(y)-h_n)-F(G(y)-h)\right|\left|H_r(y)\right|\varphi(y)dy.
\end{align*}
The first expression on the right hand side converges to $0$ by the Glivenko-Cantelli theorem and the fact that $\int \left|H_r(y)\right|\varphi(y)dy<\infty$;  
the second expression converges to $0$ due to continuity of $F$ and the dominated convergence theorem.

To show convergence of 	the  third summand note that
\begin{align*}
&\left|\int J_r(x+h)d\left(F_{\lfloor n \tau \rfloor +1, n}(x)-F(x)\right)\right|\\
&=\frac{1}{n-\lfloor n\tau\rfloor}\left|\sum\limits_{i=\lfloor n\tau\rfloor +1}^n\left(J_r(Y_i+h)-\E \ J_r(Y_i+h)\right)\right|\\
&\leq\frac{n}{n-\lfloor n\tau\rfloor}\frac{1}{n}\left|\sum\limits_{i=1}^n\left(J_r(Y_i+h)-\E \ J_r(Y_i+h) \right)\right|\\
&\quad+\frac{\lfloor n\tau\rfloor}{n-\lfloor n\tau\rfloor}\frac{1}{\lfloor n\tau\rfloor}\left|\sum\limits_{i=1}^{\lfloor n\tau\rfloor}\left(J_r(Y_i+h)-\E \ J_r(Y_i+h)\right)\right|.
\end{align*}
For both summands on the right hand side of the above inequality
the ergodic theorem implies almost sure convergence to $0$.

For the second summand in \eqref{decomposition} we have
\begin{align*}
&\frac{n-\lfloor n\tau\rfloor}{n}d_{n, r}^{-1}\lfloor n\lambda\rfloor
\int F(x+h_n)
d\left(F_{\lfloor n \tau \rfloor +1, n}-F\right)(x)\\
&=-\frac{\lfloor n\lambda\rfloor}{n}d_{n, r}^{-1}(n-\lfloor n\tau\rfloor)
\int
\left(F_{\lfloor n \tau \rfloor +1, n}(x)-F(x)\right)dF(x+h_n).
\end{align*}
Since $\frac{\lfloor n\lambda\rfloor}{n}\longrightarrow \lambda$ uniformly in $\lambda$, consider
\begin{align*}
&\Biggl|d_{n, r}^{-1}(n-\lfloor n\tau\rfloor)
\int
\left(F_{\lfloor n \tau \rfloor +1, n}(x)-F(x)\right)dF(x+h_n)\\
& \ -\frac{1}{r!}(Z_H^{(r)}(1)-Z_H^{(r)}(\tau))\int J_r(x)dF(x+h_n)\Biggr|\\
&\leq
\left|\int
d_{n, r}^{-1}n\left(F_{n}(x)-F(x)\right)-\frac{1}{r!}Z_H^{(r)}(1) J_r(x)dF(x+h)\right|\\
&\quad +\left|\int
d_{n, r}^{-1}\lfloor n\tau\rfloor\left(F_{\lfloor n\tau\rfloor}(x)-F(x)\right)-\frac{1}{r!}Z_H^{(r)}(\tau) J_r(x)dF(x+h_n)\right|\\
&\quad +\frac{1}{r!}\left|Z_H^{(r)}(1)-Z_H^{(r)}(\tau)\right|\left|\int J_r(x)d\left(F(x+h_n)-F(x+h)\right)\right|
.
\end{align*}
The first and second summand on the right hand side converge to $0$ because
of the  empirical process non-central limit theorem.
For the third summand we have
\begin{align*}
\left|\int J_r(x)d\left(F(x+h_n)-F(x+h)\right)\right|
=\left|\int \left(J_r(x-h_n)-J_r(x-h)\right)dF(x)\right|.
\end{align*}
As shown before in this proof, convergence to $0$ follows by the Glivenko-Cantelli theorem and the dominated convergence theorem.
\end{proof}

Lemma \ref{Lem:int_sq_dens} and Lemma \ref{Lem:sargmax} are needed for the proof of Theorem \ref{thm:asymp_distr}.
\begin{Lem}\label{Lem:int_sq_dens}
Suppose that Assumption \ref{ass:subordination} holds and let $l_n$, $n\in \mathbb{N}$, and $h_n$, $n\in \mathbb{N}$,  be two sequences with
 $l_n\rightarrow \infty$, $\lim_{n\rightarrow \infty}h_n =h$ and $l_n=\mathcal{O}(n)$.
Then, it holds that
\begin{align}
\sup\limits_{0\leq s \leq 1}\Biggl|& d_{l_n, r}^{-1}\lfloor l_ns\rfloor  \int\left
(F_{\lfloor l_ns\rfloor}(x+h_n)-F_{\lfloor l_ns\rfloor}(x+h)\right)dF_{n}(x)\notag\\
&-d_{l_n, r}^{-1}\lfloor l_ns\rfloor  \int\left
(F(x+h_n)-F(x+h)\right)dF(x)\Biggr|\label{Lem3:1}
\intertext{and}
\sup\limits_{0\leq s \leq 1}\Biggl| &d_{l_n, r}^{-1}\lfloor l_ns\rfloor  \int\left
(F_{n}(x+h_n)-F_{n}(x+h)\right)dF_{\lfloor l_ns\rfloor}(x)\notag\\
&-d_{l_n, r}^{-1}\lfloor l_ns\rfloor  \int\left
(F(x+h_n)-F(x+h)\right)dF(x)\Biggr|\label{Lem3:2}
\end{align}
converge to $0$ almost surely.
\end{Lem}

\begin{proof}
For the expression \eqref{Lem3:1}
the triangle inequality yields
\begin{align*}
&\sup\limits_{0\leq s \leq 1}\Biggl| d_{l_n, r}^{-1}\lfloor l_ns\rfloor  \int\left
(F_{\lfloor l_ns\rfloor}(x+h_n)-F_{\lfloor l_ns\rfloor}(x+h)\right)dF_n(x)\\
&\qquad \ \ \ -d_{l_n, r}^{-1}\lfloor l_ns\rfloor  \int\left
(F(x+h_n)-F(x+h)\right)dF(x)\Biggr|\\
&\leq 2\sup\limits_{s\in \left[0, 1\right], x\in \mathbb{R}}\left|
d_{l_n, r}^{-1}\lfloor l_ns\rfloor \left(F_{\lfloor l_ns\rfloor}(x)-F(x)\right)-\frac{1}{r!}Z_H^{(r)}(s)J_r(x)\right|\\
&\quad+\frac{1}{r!}\sup\limits_{0\leq s \leq 1}\left|Z_H^{(r)}(s)\right|
\left| \int (J_r(x+h_n)-J_r(x+h))dF_{n}(x)\right|\\
&\quad+\left|d_{l_n, r}^{-1}l_n  \int\left
(F(x+h_n)-F(x+h)\right)d\left(F_{n}-F\right)(x)\right|.
\end{align*}
The first summand converges to $0$ because of the empirical non-central limit theorem.
Moreover, 
$\sup_{0\leq s \leq 1}\left|Z_H^{(r)}(s)\right|<\infty$ a.s. due to the fact that  $Z_H^{(r)}$ is continuous with probability $1$.
It  is shown in the proof of Lemma \ref{Lem} that $\left| \int (J_r(x+h_n)-J_r(x+h))dF_{n}(x)\right|\longrightarrow 0$. As a result, the second summand vanishes as  $n$ tends to $\infty$. 

Furthermore, note that
\begin{align*}
&\left|d_{l_n, r}^{-1}l_n \int\left
(F(x+h_n)-F(x+h)\right)d\left(F_{n}-F\right)(x)\right| \\
&\leq K\left| \int \left(d_{n, r}^{-1}n\left(F_{n}(x)-F(x)\right)-\frac{1}{r!}Z_H^{(r)}(1)J_r(x)\right)dF(x+h_n)\right|\\
&\quad + K\left| \int \left(d_{n, r}^{-1}n \left(F_{n}(x)-F(x)\right)-\frac{1}{r!}Z_H^{(r)}(1)J_r(x)\right)dF(x+h)\right| \\
&\quad+K\frac{1}{r!}\left|Z_H^{(r)}(1)\right|\left|\int J_r(x)d\left(F(x+h_n)-F(x+h)\right)\right|
\end{align*}
for some constant $K$ and $n$ sufficiently large, since $l_n=\mathcal{O}(n)$.
The first and second summand on the right hand side of the above inequality converge to $0$ due to the empirical process non-central limit theorem. In addition,  we have
\begin{align*}
\left|\int J_r(x)d\left(F(x+h_n)-F(x+h)\right)\right|
=\left|\int \left(J_r(x-h_n)-J_r(x-h)\right)dF(x)\right|
\end{align*}
Therefore, it follows by the same argument as in the proof of Lemma \ref{Lem} that the third summand converges to $0$.

Considering the term in \eqref{Lem3:2}, note that
\begin{align*}
&\sup\limits_{0\leq s \leq 1}\Biggl| d_{l_n, r}^{-1}\lfloor l_ns\rfloor  \int\left
(F_{n}(x+h_n)-F_{n}(x+h)\right)dF_{\lfloor l_ns\rfloor}(x)\\
&\qquad \ \ -d_{l_n, r}^{-1}\lfloor l_ns\rfloor  \int\left
(F(x+h_n)-F(x+h)\right)dF(x)\Biggr|\\
&\leq 2\sup\limits_{0\leq s \leq 1, x\in \mathbb{R}}\left| d_{l_n, r}^{-1}\lfloor l_ns\rfloor \left(F_{\lfloor l_ns\rfloor}(x)-F(x)\right)-\frac{1}{r!}Z_H^{(r)}(s)J_r(x)\right|\\
& \quad+\frac{1}{r!}\sup\limits_{0\leq s \leq 1}\left|
Z_H^{(r)}(s)\right| \left|\int J_r(x)d\left(F_n(x+h_n)-F_n(x+h)\right)\right|\\
&\quad+2K\sup\limits_{ x\in \mathbb{R}}\left| d_{n, r}^{-1}n \left
(F_n(x)-F(x)\right)-\frac{1}{r!}Z_H^{(r)}(1)J_r(x)\right|\\
&\quad+\frac{1}{r!}\left| Z_H^{(r)}(1)\right|\int\left|J_r(x+h_n)-J_r(x+h)\right|dF(x)
\end{align*}
for some constant $K$ and $n$ sufficiently large.
The first and third summand on the right hand side of the above inequality converge to $0$ due to the empirical process non-central limit theorem.  The last summand converges to $0$ due to the corresponding argument in the proof of Lemma  \ref{Lem}.
It holds that
\begin{align*}
& \left|\int J_r(x)d\left(F_n(x+h_n)-F_n(x+h)\right)\right|\\
&=\left|\int \left(F_n(G(y)-h_n)-F_n(G(y)-h)\right)H_r(y)\varphi(y)dy\right|\\
&\leq \left(2\sup\limits_{x \in \mathbb{R}}\left|F_n(x)-F(x)\right|+\sup\limits_{x \in \mathbb{R}}\left|F(x-h_n)-F(x-h)\right|\right)\int \left|H_r(y)\right|\varphi(y)dy.
\end{align*}
The right hand side of the above inequality converges to $0$ almost surely due to  the Glivenko-Cantelli theorem and because  $F$ is uniformly continuous. As a result, the second summand converges to $0$, as well.
\end{proof}

Lemma \ref{Lem:sargmax} establishes a condition
under which convergence in distribution of a sequence of random variables  entails convergence of the smallest argmax of the sequence. 

\begin{Lem}\label{Lem:sargmax}
Let $K$ be a compact interval and denote by $D(K)$  the corresponding Skorohod space, i.e. the collection of all functions $f:K\longrightarrow\mathbb{R}$ which are  right-continuous with left limits.
Assume that $Z_n$, $n\in \mathbb{N}$,  are random variables taking values in $D(K)$ and that $Z_n\overset{\mathcal{D}}{\longrightarrow}Z$, where (with probability $1$) $Z$ is continuous and $Z$ has a unique maximizer. Then $\sargmax(Z_n)\overset{\mathcal{D}}{\longrightarrow}\sargmax(Z)$.
\end{Lem}

\begin{proof}
Due to Skorohod's representation theorem there exist random variables $\tilde{Z}_n$ and $\tilde{Z}$ defined on a common probability space $(\tilde{\Omega}, \tilde{F}, \tilde{P})$, such that $\tilde{Z}_n\overset{\mathcal{D}}{=}Z_n$, $\tilde{Z}\overset{\mathcal{D}}{=}Z$ and $\tilde{Z}_n\overset{a.s.}{\longrightarrow}\tilde{Z}$. Due to Lemma 2.9 in \cite{SeijoSen2011} the smallest argmax functional is continuous at $W$ (with respect to the Skorohod-metric and  the sup-norm metric) if $W\in D(K)$ is a continuous function which has a unique maximizer. 
Since (with probability $1$) $Z$ is continuous with unique maximizer, $\sargmax(\tilde{Z}_n)\overset{a.s.}{\longrightarrow}\sargmax(\tilde{Z})$.
As almost sure convergence implies convergence in distribution, we have $\sargmax(\tilde{Z}_n)\overset{\mathcal{D}}{\longrightarrow}\sargmax(\tilde{Z})$ and therefore $\sargmax(Z_n)\overset{\mathcal{D}}{\longrightarrow}\sargmax(Z)$.
\end{proof}

\end{document}